\documentclass[final]{siamltex}

\usepackage{graphicx} 
\usepackage{url}      
\newcommand{\doi}[1]{\url{http://dx.doi.org/#1}}
\usepackage{amsmath}  
\usepackage{amssymb}
\usepackage{mathrsfs}

\usepackage{xspace,algorithm,algorithmic,tabularx}

\usepackage{setspace}
\usepackage{color}
\usepackage{pdflscape}
\usepackage{graphics}
\usepackage{rotating}
\usepackage{verbatim}
\usepackage{amsmath}
\usepackage{amssymb}
\usepackage{amsfonts}
\usepackage{mathrsfs}
\usepackage{array}
\usepackage{epic,eepic}
\usepackage{fancyhdr}
\usepackage{url}
\usepackage{subfigure}

\newtheorem{remark}{Remark}[section]
\newtheorem{assumptions}{Assumptions}[section]

\DeclareMathOperator{\decay}{decay}
\DeclareMathOperator{\tail}{tail}
\DeclareMathOperator{\cost}{cost}
\DeclareMathOperator{\var}{var}
\DeclareMathOperator{\fix}{fix}

\def\seqv{\boldsymbol{\rm v}}
\def\bsa{\boldsymbol{a}}

\def\bsx{\boldsymbol{x}}
\def\bsy{\boldsymbol{y}}

\def\bsk{\boldsymbol{k}}
\def\bsl{\boldsymbol{l}}
\def\bsh{\boldsymbol{h}}
\def\bs0{\boldsymbol{0}}
\def\bsq{\boldsymbol{q}}

\def\bst{\boldsymbol{t}}

\def\bsgamma{\boldsymbol{\gamma}}

\def\bsqs12{\boldsymbol{q}_{[s_2] \setminus [s_1]}}
\def\Gs12{G^{s_2-s_1}_{b,m}}

\newcommand{\X}{{\mathfrak X}}
\newcommand{\U}{{\mathcal U}}
\newcommand{\N}{\mathbb{N}}
\newcommand{\E}{\mathbb{E}}
\newcommand{\R}{\mathbb{R}}
\newcommand{\EE}{\mathbb{E}}
\newcommand{\PP}{\mathbb{P}}
\def\nn{\mathbb{N}}
\def\zz{\mathbb{Z}}

\newcommand\Var{\textnormal{Var}}

\newcommand{\rr}{\mathbb{R}}

\newcommand{\bszero}{\boldsymbol{0}}

\allowdisplaybreaks

\title{
       Optimal randomized multilevel algorithms for infinite-dimensional integration on function spaces with ANOVA-type decomposition}

\author{Jan Baldeaux\thanks{Finance Discipline Group, University of Technology, Sydney ({\tt Jan.Baldeaux@uts.edu.au}).}
        \and Michael Gnewuch\thanks{School of Mathematics and Statistics,
University of New South Wales ({\tt mig@numerik.uni-kiel.de}).}}

\date{}

\begin{document}
\maketitle

\begin{abstract}
In this paper, we consider the infinite-dimensional integration problem on
weighted reproducing kernel Hilbert spaces with norms induced by an underlying function space decomposition of ANOVA-type. The weights model the relative importance of different groups of variables. We present new randomized multilevel algorithms to tackle this integration problem and prove upper bounds for their
randomized error.
Furthermore, we provide in this setting the first non-trivial lower error bounds for general randomized algorithms,
which, in particular, may be adaptive or non-linear. These lower bounds show that our multilevel algorithms are optimal. Our analysis refines and extends
the analysis provided in [F.~J.~Hickernell, T.~M\"uller-Gronbach, B.~Niu, K.~Ritter,
J.~Complexity 26 (2010), 229--254], and our error bounds improve substantially on the error bounds
presented there. As an illustrative example, we discuss the unanchored Sobolev space and employ randomized quasi-Monte Carlo multilevel algorithms based on scrambled polynomial lattice rules.
\end{abstract}

 \vspace*{2cm}

\section{Introduction} \label{introduction}

%

Motivated by applications arising e.g. in quantitativ finance or physics, see \cite{Gil08a, WW96}, there has recently been done a large amount of research investigating integrals over functions with apriori unlimited or even infinitely many variables. \emph{Multilevel algorithms}, \cite{Hei98, Gil08a}, have been successfully used to solve these problems. Furthermore, multilevel algorithms have been successfully combined with QMC methods, \cite{GiW09}, which turned
out to be more efficient than plain Monte Carlo (MC) or quasi-Monte Carlo (QMC) algorithms, respectively.

Researchers in information-based complexity started to study the complexity of the
infinite-dimensional integration problem on weighted reproducing kernel Hilbert
spaces of integrands with norms induced by function space decompositions of
\emph{anchored} or \emph{ANOVA-type} \cite{HW01, KSWW10a, HMGNR10, NHMR11, G10, PW11, B10}.
(For function space decompositions we refer to \cite{KSWW10b}.)
For many spaces of integrands good lower bounds for the deterministic worst-case
integration error have been proved and constructive upper bounds for different
error criteria have been established with the help of multilevel \cite{HMGNR10, NHMR11, G10, B10}
and so-called \emph{changing dimension algorithms} \cite{KSWW10a, PW11}.
For some settings these bounds are sharp.
Nevertheless, the randomized setting and the case of  function spaces with norms induced by ANOVA-type decompositions are so far not well enough
understood; see also the comments in \cite{HMGNR10} or, for integration in the
randomized setting in general, in \cite[p.~487]{NW10}. The main reason for this is
that the randomized setting and the ANOVA setting are technically demanding and
more difficult to analyze than the deterministic worst-case setting and the anchored setting.
But the former two settings are particularly
interesting and very important. The (deterministic) worst-case error is often unnecessarily pessimistic and furthermore
suitably randomized algorithms can achieve higher convergence rates and additionally
provide  statistical error estimates. ANOVA decompositions have been used to explain the success of QMC methods for financial applications,
see e.g. \cite{PT95, CMO97, TW98}: If the effective dimension, see \cite{CMO97}, of the integration problem is small, i.e. the variance is concentrated in the lower-order ANOVA terms,
QMC methods can be expected to perform well. Furthermore, in \cite{LO06} and \cite{GKS10, GKS11} it was shown that lower order ANOVA terms exhibit more smoothness than the corresponding function itself.

In \cite{HMGNR10} the convergence rates of randomized multilevel algorithms for
infinite-dimen\-sional integration on Hilbert spaces with product weights
are analyzed. But as the authors admit in their paper, in the ANOVA case their analysis has unfortunately some limitations.
An undesirable consequence of this shortcoming is that they are only able
to study a very restricted class of multilevel algorithms.
Non-trivial lower bounds for the errors of randomized multilevel algorithms
are not provided in \cite{HMGNR10}.

In this paper we refine the analysis from \cite{HMGNR10}
and extend it to other kinds of weights.
As a result we are able to study new multilevel algorithms
and to establish good upper error bounds for their performance.
In the case of product weights our upper error bounds improve substantially
on the ones given in \cite{HMGNR10} and \cite{B10}. A key indegredient
for our analysis of multilevel algorithms is the ``ANOVA invariance lemma'',
Lemma \ref{ANOVA}.
We also provide the first non-trivial lower bounds for the $N$th minimal errors of
randomized multilevel algorithms (or, to be more precise, of general randomized
algorithms in the \emph{variable subspace sampling model} introduced in
\cite{CDMR09}; for lower error bounds for the  $N$th minimal errors of deterministic and randomized algorithms
in the case of anchored decompositions in the former model and the cost model introduced in \cite{KSWW10a} we refer to the new preprints \cite{DG12, Gne12}). These lower bounds show that our constructive
upper bounds are tight for both types of weights considered, namely finite-intersection weights and product weights. (Similar optimal results for multilevel algorithms are
achieved  in \cite{DG12}
in the deterministic worst-case setting for norms based on anchored function space decompositions.)  Furthermore, as done in \cite{HMGNR10} for product weights, we provide for finite-intersection weights sharp upper and lower error bounds for single-level algorithms (or, to be more precise, upper bounds for specific
and lower bounds for general randomized algorithms in the \emph{fixed subspace
sampling model} defined in \cite{CDMR09}). Our analysis tools can also be used to investigate the convergence rates of
other randomized algorithms, as, e.g., the randomized changing dimension algorithms
from \cite{PW11},
in the ANOVA setting.

The paper is organized as follows: 
In Section \ref{Preliminaries}, we recall preliminaries, but also provide new lemmas which are important for our error analysis.
In Section \ref{LOWBOU} we provide lower bounds for the randomized errors
of general randomized algorithms and general weights. We specify these bounds
for finite-intersection and product weights.
In Section \ref{secmultilevelalg} we present our multilevel algorithms for general
weights and provide concrete error bounds for finite-intersection weights in
Theorem \ref{Error-FIW} and for product weights in Theorem \ref{theoprodweightsml}. In Section \ref{secexample} we consider a concrete space of functions of
infinitely many variables 
and show that multilevel algorithms based on \emph{scrambled polynomial lattice rules} are essentially optimal
for finite-intersection and product weights.

\section{Preliminaries} \label{Preliminaries}

Let us make some remarks on notation: For $n\in\N$ we denote by $[n]$ the
set $\{1,2,\ldots,n\}$. For a finite set $u$ we denote its cardinality by
$|u|$. We use the common Landau $O$-notation. For two functions $f$ and $g$
we write occasionally $f=\Omega(g)$ for $g= O(f)$, and $f=\Theta(g)$ if $f=\Omega(g)$ and $f=O(g)$ holds.
If we consider a reproducing kernel $K$, then we always denote the
corresponding reproducing kernel Hilbert space by $H(K)$ and its norm unit
ball by $B(K)$. The norm and the scalar product of $H(K)$ are denoted by
$\|\cdot\|_K$ and $\langle \cdot, \cdot \rangle_K$, respectively.
Our standard reference for reproducing kernel Hilbert spaces
is \cite{Aro50}.

\subsection{The ANOVA decomposition}\label{SECT2.1}

In this section, we recall the ANOVA decomposition of $L^2$-functions; the acronym ``ANOVA'' stands
for ``Analysis of Variance''.
Let $(\Omega, \Sigma, \PP)$ be a probability space, and denote its $d$-fold product space by $(\Omega^d, \Sigma^d, \PP^d)$.
The ANOVA decomposition of an $L^2$-function $f:\Omega^d \to \R$ is
\begin{equation} \label{eqANOVAdecomp}
f(\bsx) = \sum_{u \subseteq [d]} f_u(\bsx) \, ,
\end{equation}
where $f_u$ denotes the ANOVA-term corresponding to the set $u$.
For $u\subseteq [d]$ and $\bsx\in \Omega^d$ let $\bsx_u := (x_j)_{j\in u}
\in \Omega^u$.
For $\bsx_u\in \Omega^u$ and $\omega\in \Omega^{[d]\setminus u}$ let
$(\bsx_u,\omega)\in \Omega^d$ be the vector whose $j$th component is
$x_j$ if $j\in u$ and $\omega_j$ otherwise.
The ANOVA-term $f_u$ can be computed recursively via
\begin{displaymath}
f_u(\bsx) = \int_{\Omega^{[d] \setminus u}} f(\bsx_u,\omega) \,\PP^{[d]\setminus u} (d\omega) - \sum_{v \subsetneq u} f_v(\bsx) \, , \textrm{ where } f_{\emptyset} = \int_{\Omega^d} f(\omega) \,\PP^d(d \omega).
\end{displaymath}
Furthermore, it can be shown via induction over $|u|$ that
\begin{equation} \label{eqANOVAdecompprop}
\int_{\Omega} f_u(\bsx) \,\PP(d x_j) = 0
\hspace{2ex}\text{for all $j \in u$.}
\end{equation}
The important feature of the ANOVA decomposition is
\begin{equation}
 \label{anova}
\Var(f) = \sum_{u\subseteq [d]} \Var(f_u).
\end{equation}

Let $(D,\Sigma',\rho)$ be another probability space.
The new randomized algorithms for infinite-dimensional integration we
present here, rely on random quadratures that use $n$ (deterministic) real coefficients $w_i$ and $n$ randomly chosen quadrature points
$\bsx^{(1)}(\omega),\ldots, \bsx^{(n)}(\omega)$ in $D^d$, i.e., that have the
form
\begin{displaymath}
Q_n(\omega,f) = \sum^n_{i=1} w_i f(\bsx^{(i)} (\omega))
\end{displaymath}
for $f\in L^2(D^d, \rho^d)$, $\omega \in \Omega^d$. We assume that
for fixed $f$  the function $\omega \mapsto Q_n(\omega,f)$ is square integrable.
The next lemma is crucial for the proof of our upper error bounds for multilevel algorithms; it says that under a certain condition the $u$th ANOVA-term
of the $L^2(\Omega^d,\PP^d)$-function  $Q_n(\cdot,f)$ is equal to $Q_n$
applied to the $u$th ANOVA-term of the $L^2(D^d, \rho^d)$-function $f$.
We denote the ANOVA-terms of $Q_n(\cdot,f)$, regarded as a function on $\Omega^d$, by $\left[ Q_n(\cdot, f) \right]_u$, $u\subseteq [d]$.

\begin{lemma}[ANOVA Invariance Lemma]
\label{ANOVA}
Let $(\Omega, \Sigma, \PP)$, $(D,\Sigma', \rho)$ be probability spaces.
Let $d\in\N$, and let $\mathcal{V}$ be a subset of the power set of $[d]$.
Assume that $Q_n = Q_{[d],n}$, given by
\begin{equation}
\label{alg-form}
Q_n(\omega, f) = \sum^n_{i=1} w_i f(\bsx^{(i)}(\omega)),
\hspace{3ex}\omega\in\Omega^d,\, w_1,\ldots, w_n\in\R, \,f\in L^2(D^d, \rho^d),
\end{equation}
is a randomized linear algorithm which
satisfies the following condition:
\begin{itemize}
\item[(*)] For all $v\in \mathcal{V}$ the random points $\bsx^{(i)}_v = (x^{(i)}_j(\omega))_{j\in v} \in D^v$,
$i=1,\ldots,n$,  are of the form
$\bsx^{(i)}_v = (x^{(i)}_j(\omega_j))_{j\in v}$, and
the random variables $x^{(i)}_j$
are distributed according to the law $\rho$.
\end{itemize}
Then we have for each $f\in L^2(D^d, \rho^d)$ whose ANOVA terms $f_v$ vanish if
$v\notin \mathcal{V}$ that
\begin{equation}
\label{vertauschungsrelation}
[Q_n(\cdot, f)]_u = Q_n(\cdot, f_u)
\hspace{3ex}\text{for all $u\subseteq [d]$.}
\end{equation}
\end{lemma}

\begin{proof}
We prove (\ref{vertauschungsrelation}) by induction on $|u|$. So let first $u=\emptyset$. Then, due to (\ref{alg-form}), (\ref{eqANOVAdecomp}), and condition (*),
\begin{equation*}
\begin{split}
[Q(\cdot,f)]_\emptyset &= \int_{\Omega^d} Q(\omega, f)\,\PP^d ({\rm d}\omega)
= \sum^n_{i=1} w_i \int_{\Omega^d} f(\bsx^{(i)}(\omega)) \,\PP^d ({\rm d}\omega)\\
&= \sum^n_{i=1} w_i \sum_{u\subseteq [d]} \int_{\Omega^d} f_u(\bsx^{(i)}(\omega)) \,\PP^d ({\rm d}\omega)
= \sum^n_{i=1} w_i \sum_{u\subseteq [d]} \int_{D^d} f_u(\bsx) \,{\rm d}\rho^d(\bsx)\\
&= \sum^n_{i=1} w_i \int_{D^d} f(\bsx)\,\rho^d(\bsx) = Q(\cdot, f_\emptyset).
\end{split}
\end{equation*}
Let now $\emptyset \neq v \subseteq [d]$, and let us assume that
(\ref{vertauschungsrelation}) holds for all $u$ with $|u| < |v|$.
Then we have for $\sigma \in \Omega^d$
\begin{equation*}
[Q(\cdot,f)]_v(\sigma) = \int_{\Omega^{[d]\setminus v}}
Q \big( (\sigma_v,\omega),
f \big) \,\PP^{[d]\setminus v}({\rm d}\omega)
- \sum_{u\subsetneq v} [Q(\cdot,f)]_u(\sigma).
\end{equation*}
Now
\begin{equation*}
\begin{split}
\int_{\Omega^{[d]\setminus v}}
Q \big( (\sigma_v,\omega),
f \big) \,\PP^{[d]\setminus v}({\rm d}\omega)
&= \sum^n_{i=1} w_i \int_{\Omega^{[d]\setminus v}} f(\bsx^{(i)}(\sigma_v,\omega))
\,\PP^{[d]\setminus v}({\rm d}\omega)\\
&= \sum^n_{i=1} w_i \sum_{u\subseteq [d]} \int_{\Omega^{[d]\setminus v}} f_u(\bsx^{(i)}(\sigma_v,\omega))
\,\PP^{[d]\setminus v}({\rm d}\omega).
\end{split}
\end{equation*}
Notice that the last integral is zero if $u$ is not a subset of $v$,
due to condition (*) and (\ref{eqANOVAdecompprop}) for $u\in \mathcal{V}$ and due to
$f_u=0$ for $u\notin \mathcal{V}$. Since the ANOVA terms $f_u$, $u\subseteq v$, depend
only on the variables in $v$, we thus get
\begin{equation*}
\begin{split}
\int_{\Omega^{[d]\setminus v}}
Q \big( (\sigma_v,\omega),
f \big) \,\PP^{[d]\setminus v}({\rm d}\omega)
&= \sum_{i=1}^n w_i \sum_{u\subseteq v} f_u(\bsx^{(i)}(\sigma))
= \sum_{u\subseteq v} Q(\sigma, f_u)\\
&= Q(\sigma,f_v) + \sum_{u\subsetneq v} [Q(\cdot,f)]_u(\sigma),
\end{split}
\end{equation*}
where the last step uses the induction hypothesis. Hence
$[Q(\cdot,f)]_v(\sigma) = Q(\sigma, f_v)$, and the proof is complete.
\end{proof}

\begin{remark}\label{Unbiased}
In the case where the set $\mathcal{V}$ in Lemma \ref{ANOVA} is the whole
power set of $[d]$, we may say that $Q_n$ is invariant under the ANOVA decomposition.
Note that for general subsets $\mathcal{V}$ of the power
set of $[d]$ and $f\in L^2(D^d,\rho^d)$ with $f_v=0$ for all
$v\notin \mathcal{V}$ condition (*) of Lemma \ref{ANOVA} implies that
$Q_n(\cdot,f)$ is square integrable on
$\Omega^d$ and, if additionally
\begin{equation}\label{summe=1}
\sum_{i=1}^n w_i = 1
\end{equation}
holds, an unbiased estimator of
$\int_{D^d}f(\bsx) \,\rho^d({\rm d}\bsx)$.
\end{remark}

\subsection{Classes of weights}

Let
\begin{equation*}
\U := \{u\subset \nn \,|\, |u|<\infty\},
\end{equation*}
and let $\bsgamma=(\gamma_u)_{u\in \U}$ be a sequence of non-negative weights. The weights $\bsgamma$ are called \emph{product weights}, \cite{SW98}, if there exists a sequence of non-negative numbers $\gamma_1 \ge \gamma_2 \ge \cdots$ such
that $\gamma_u = \prod_{j\in u} \gamma_j$ for all $u \in \U $.
The weights $\bsgamma$ are called \emph{finite-order weights}, \cite{DSWW06}, of order $\omega$
if $\gamma_u = 0 \hspace{2ex}\text{for all $u\in\U$ with $|u|>\omega$.}$
We are particularly interested in some subclass of finite-order
weights. We restate
Definition 3.5 from \cite{G10}.

\begin{definition}
\label{def-fiw}
Let $\rho \in \nn$.
Finite-order weights $(\gamma_{u})_{u\in \U}$ are
called \emph{finite-intersection weights} with \emph{intersection
degree} at most $\rho$ if we have
\begin{equation}
\label{fiw}
|\{v\in\U \, | \, \gamma_v >0 \,,\, u\cap v \neq \emptyset \}| \le 1+\rho
\hspace{2ex}\text{for all $u\in\U$ with $\gamma_u >0$.}
\end{equation}
\end{definition}
Note that for finite-order weights of order $\omega$ condition (\ref{fiw}) is
equivalent to the following
condition: There exists an $\eta\in\nn$ such that
\begin{equation}
\label{cond}
|\{ u\in\U \,|\, \gamma_u >0 \,,\, k\in u \}| \le \eta
\hspace{2ex}\text{for all $k\in\nn$.}
\end{equation}
Indeed, if (\ref{fiw}) is satisfied, then (\ref{cond}) holds with
$\eta \le 1+\rho$, and if (\ref{cond}) is satisfied, then
(\ref{fiw}) holds with $\rho \le (\eta-1)\omega$. A subclass of the finite-intersection weights are the
\emph{finite-diameter weights} proposed by Creutzig, see, e.g., \cite{G10,NW08}. Let us restate Lemma 3.10 from \cite{G10}, which will be essential
for our analysis of finite-intersection weights.
\begin{lemma}
\label{phi}
Let $(\gamma_{u})_{u\in\U}$ be finite-intersection weights of
finite order $\omega$. Let $\eta\in\N$ be such that (\ref{cond})
is satisfied. Then there exists a
mapping $\phi: \N \to [\eta(\omega-1) + 1]$ such that for all $u\in\U$
with $\gamma_u>0$
the restriction $\phi|_{u}$ is injective.
\end{lemma}

\subsection{Function Spaces} \label{subsecfunspace}

Let $D\subseteq \rr$, $\rho$ a probability measure on $D$, and
$\mu:=\otimes_{n\in\nn}\, \rho$ the product probability measure on
$D^\nn$.
Unless stated otherwise, we denote by $u$, $v$, and $w$ finite subsets
of $\N$, i.e., $u,v,w\in \U$. In many formulas we will not state this
explicitly, to make our notation not too cumbersome.
Let $(\gamma_u)_{u\in \U}$ be a sequence of
non-negative weights.

In this paper we make essentially the same assumptions as in
\cite[Sect.~2]{HMGNR10}.

\begin{assumptions}
 We assume that
\begin{itemize}
\item[{\rm (A 1)}] $k\neq 0$ is a measurable reproducing kernel on $D\times D$
which satisfies
\item[{\rm (A 2)}] $H(k)\cap H(1) = \{0\}$ as well as
\item[{\rm (A 3)}] $M:= \int_D k(x,x) \rho({\rm d}x) < \infty$.
\item[{\rm (A 4)}] $\gamma_\emptyset = 1$ and
\begin{equation}
\label{summable}
\sum_{u\in \U} \gamma_u M^{|u|} <\infty.
\end{equation}
\end{itemize}
\end{assumptions}
It is easily verified that for product weights and finite-order weights condition (\ref{summable}) can be replaced by the equivalent condition
$\sum_{u\in \U} \gamma_u < \infty$.

For $u\in \U$ we put $k_u(\bsx,\bsy) := \prod_{j\in u} k(x_j,y_j)$, for all $\bsx, \bsy\in D^\nn$.
In particular, $k_\emptyset(\bsx,\bsy) = 1$.
We define $H_u := H(k_u)$, i.e., $H_u$ is the reproducing kernel
Hilbert space with kernel $k_u$. The following lemma stems from \cite{HMGNR10}.
\begin{lemma}
\label{restrict}
Let $\bsx,\bsy\in D^{\N}$ and $f\in H_u$. If $\bsx_u = \bsy_u$,
then $f(\bsx) = f(\bsy)$.
\end{lemma}

Given $v\in\mathcal{U}$ we define the \emph{weighted kernel} $K_v(\bsx,\bsy) := \sum_{u\subseteq v} \gamma_u k_u(\bsx,\bsy)$, for $\bsx,\bsy\in D^{\nn}$.
For the next lemma see \cite[Lemma~3]{HW01} or \cite[I, \S~6]{Aro50}.
\begin{lemma}\label{Lemma5}
The reproducing kernel Hilbert space $H(K_v)$ consists of all
functions $f = \sum_{u\subseteq v} f_u$, for $f_u\in H_u$.
Furthermore, $\|f\|^2_{K_v} = \sum_{u\subseteq v} \gamma^{-1}_u \|f_u\|^2_{k_u}.$
\end{lemma}

In general we follow the convention that $\infty\cdot 0 = 0$. Note that
$\gamma_u = 0$ implies $f_u=0$ for all $f\in H(K_v)$;  in that
case  $\gamma^{-1}\|f_u\|^2_{k_u} = 0$.

Due to Lemma \ref{restrict} we can consider the spaces $H(k_u)$
and $H(K_u)$ as spaces of functions on $D^u$.
In this case we have
$H(k_u) = \otimes_{j\in u} H(k)$, and $H(K_u)$ is a tensor product
space if and only if the weights $(\gamma_u)_{u\in \mathcal{U}}$
are product weights, see, e.g., \cite[I, \S~8]{Aro50}.

Let us define the domain $\X$ of functions of infinitely many variables by
\begin{equation}
 \X := \bigg\{ \bsx\in D^{\N} \,\bigg|\, \sum_{u\in \U} \gamma_u \prod_{j\in u} k(x_j, x_j) <\infty \bigg\}.
\end{equation}
Similar as in \cite[Lemma~1]{HMGNR10} one shows that $\X$ satisfies
$\mu(\X) = 1$.
For $\bsx,\bsy \in\X$ we put
\begin{equation*}
K(\bsx,\bsy) := \sum_{u\in\mathcal{U}} \gamma_u k_u(\bsx,\bsy).
\end{equation*}
Since $K$ is well-defined, symmetric, and positive
semi-definite, it is a reproducing
kernel on $\X \times \X$, see \cite{Aro50}. For the next lemma see \cite[Cor.~5]{HW01} or \cite{GMR12}.
\begin{lemma}
\label{Lemma6}
The reproducing kernel Hilbert space $H(K)$ consists of all functions
$f=\sum_{u\in\mathcal{U}} f_u$, $f_u\in H_u$, such that
\begin{equation*}
\sum_{u\in\mathcal{U}} \gamma^{-1}_u \|f_u\|^2_{k_u} <\infty.
\end{equation*}
In the case of convergence, we have
\begin{equation}\label{norm-formel}
\|f\|^2_K = \sum_{u\in\mathcal{U}} \gamma^{-1}_u \|f_u\|^2_{k_u}.
\end{equation}
\end{lemma}

If $f\in H(K)$, then the decomposition
\begin{equation}\label{function-decomp}
f=\sum_{u\in \mathcal{U}} f_u,
\hspace{3ex}f_u \in H_u,
\end{equation}
is uniquely defined, since $f_u$ is the orthogonal
projection of $f$ onto $H_u$.

\subsection{Integration}

Integration with respect to the probablitiy measure $\mu$
defines a bounded linear functional
\begin{equation*}
I(f) := \int_{\X} f(\bsx)\, \mu({\rm d}\bsx)
\end{equation*}
on $H(K)$, as verified by the following estimates:
\begin{equation*}
\int_{\X} |f(\bsx)|\, \mu({\rm d}\bsx) =
\int_{\X} |\langle f, K(\cdot, \bsx) \rangle_K|
\, \mu({\rm d}\bsx) \le \|f\|_K \int_{\X} \|K(\cdot, \bsx)\|_K
\, \mu({\rm d}\bsx),
\end{equation*}
and
\begin{equation*}
\begin{split}
\left( \int_{\X} \|K(\cdot, \bsx)\|_K \,\mu({\rm d}\bsx) \right)^2
\le \int_{\X} \|K(\cdot, \bsx)\|_K^2 \,\mu(d \bsx)
= \int_{\X} K(\bsx, \bsx) \,\mu({\rm d}\bsx)
\le \sum_{u\in\U} \gamma_u M^{|u|},
\end{split}
\end{equation*}
and the last term is finite due to (\ref{summable}). The representer $h\in H(K)$ of the integration functional $I$ is given by
\begin{equation}
\label{representer}
h(\bsx) = \langle h, K(\cdot, \bsx) \rangle_K
= \int_{\X} K(\bsx,\bsy)\, \mu({\rm d}\bsy).
\end{equation}
Similar as above, it is easily shown that for every $u\in H_u$
\begin{equation*}
I_u(f) := \int_{D^u} f(\bsx)\, \rho^u({\rm d}\bsx)
\end{equation*}
defines a bounded linear functional on $H_u$. It is also easily shown that
$H(K) \subset L^2(\X,\mu)$ and
$H_u \subset L^2(D^u, \rho^u)$ for all $u\in\U$. For the rest of this article we assume that the following
assumptions hold:
\begin{assumptions}
We assume that
\begin{itemize}
 \item[{\rm (A~2a)}]
$\int_D k(x,y)\,\rho({\rm d}x) = 0$
for all $y\in D$.

 \item[{\rm (A~5)}]
For all $a\in D$ we have $k(a,a)>0$.
\end{itemize}
\end{assumptions}
Note that assumption (A~2a) and identity (\ref{representer}) immediately imply that
\begin{equation}
\label{h=1}
h(\bsx) = 1 \hspace{2ex}\text{for all $\bsx\in \X$.}
\end{equation}
Thus, if there exists an $a^*\in D$ with $k(a^*,a^*) = 0$, then this
results for $\bsa^*:=(a^*)_{j\in\N}$ in $K(\cdot,\bsa^*) = h$,
which leads to $I(f) = f(\bsa^*)$ for all $f\in H(K)$.
Assumption (A~5) avoids this trivial integration problem.

Under assumption (A~2a), the uniquely determined decomposition (\ref{function-decomp})
is in fact the ANOVA decomposition of $f$, see Remark \ref{YES-ANOVA!}.

\subsection{Projections}\label{PROJECTIONS}

Let us choose an anchor $\bsa\in \X$. Here the most interesting case
seems to us a vector $\bsa$ whose entries are all equal to $a$, where
$a\in D$ satisfies
\begin{equation}
 \label{anker_a}
\sum_{u\in \U} \gamma_u k(a,a)^{|u|} < \infty;
\end{equation}
note that
condition (\ref{summable}) ensures that such an $a$ exists.
For the sake of generality we will consider a general $\bsa \in \X$.
But to make proofs not unnecessarily complicated, we will restrict
ourselves to anchors $\bsa = (a,a,\ldots) \in \X$ for the concrete
analysis of our constructive multilevel algorithms in the case of product weights and
finite-intersection weights.

We define for $u\in \mathcal{U}$
\begin{equation*}
(\Psi_{u,\bsa}f)(\bsx) := f(\bsx_u; \bsa)
\hspace{2ex}\text{for all $\bsx\in \X$,}
\end{equation*}
where $(\bsx_u;\bsa) := (\bsx_u, \bsa_{\N\setminus u})$.
Note that due to (\ref{anker_a}) we have $(\bsx_u;\bsa) \in \X$.

For $u, v, w \in\mathcal{U}$ with $u\subseteq v \subset w$ we define
\begin{equation}
\label{f+-}
f_{u,v}^+ := \sum_{u' \subset \N\setminus v} f_{u\cup u'}
\hspace{2ex}\text{and}\hspace{2ex}
f_{u,v,w}^- := \sum_{u' \subset \N\setminus v \,;\, u'\cap w \neq \emptyset} \,
f_{u\cup u'},
\end{equation}
as well as
\begin{equation}
\label{r}
r^2_{v,u,\bsa} :=
\sum_{u'\subset \N\setminus v}
\gamma_{u\cup u'}\, k_{u'}(\bsa,\bsa)
\end{equation}
and
\begin{equation}
\label{r_tilde}
\tilde{r}^2_{w,v,u,\bsa} :=
\sum_{u'\subset \N\setminus v\,;\,
u'\cap w \neq \emptyset}
\gamma_{u\cup u'}\, k_{u'}(\bsa,\bsa).
\end{equation}
Since $\bsa\in\X$, the quantities $r_{v,u,\bsa}$ and $\tilde{r}_{w,v,u,\bsa}$ are
finite. Furthermore, we have $\tilde{r}^2_{w,v,u,\bsa} \le r^2_{v,u,\bsa} - \gamma_{u}$.

\begin{remark}\label{Obs}
Observe that we have the orthogonal decomposition
\begin{equation}
 \label{obs}
f = \sum_{u\subseteq v} f^+_{u,v}
\hspace{2ex}\text{for all $f\in H(K)$ and all $v\in\U$.}
\end{equation}
For a fixed $f \in H(K)$ and $v \subset w$ the functions $f^{-}_{u,v,w}$, $u \in v$, form an orthogonal function system in $H(K)$.
\end{remark}

\begin{lemma}
\label{Analogon-Lemma7}
For all $f\in H(K)$ and all finite subsets
$u\subseteq v \subset w$ of $\N$ we have
$\Psi_{v,\bsa}(f^+_{u,v}), \Psi_{v,\bsa}(f^-_{u,v,w}) \in H_u$ and
the norm estimates
\begin{equation}
\label{norm_psi_k}
\| \Psi_{v,\bsa}
(f^+_{u,v}) \|_{k_{u}}
\le r_{v,u,\bsa} \|  f^+_{u,v} \|_{K}
\end{equation}
and
\begin{equation}
\label{norm_psi_k-1}
\| \Psi_{v,\bsa} (f^-_{u,v,w}) \|_{k_{u}}
\le \tilde{r}_{w,v,u,\bsa} \| f^-_{u,v,w} \|_{K}.
\end{equation}
Furthermore, $\Psi_{v,\bsa}$ is a bounded
projection from $H(K)$ onto $H(K_v)$, and its operator norm is given by
\begin{equation}
\label{op_norm_pro}
 \|\Psi_{v,\bsa}\|_{K\to K_v} = \max_{u\subseteq v \,;\, \gamma_u>0} \gamma^{-1/2}_u r_{v,u,\bsa}.
\end{equation}
\end{lemma}

\begin{proof}
To prove (\ref{norm_psi_k}), we apply \cite[Lemma 15]{HMGNR10}: Put
\begin{equation*}
E_1:=D^{v}
\hspace{2ex}\text{and}\hspace{2ex}
E_2:= \bigcap_{u\subseteq v} \bigg\{\bsx\in D^{\N\setminus v}
\, \bigg|\, \sum_{u'\subset \N\setminus v} \gamma_{u\cup u'}
\prod_{j\in u'} k(x_j,x_j) < \infty \bigg\}.
\end{equation*}
Since we assumed that there exists no $a^* \in D$ with $k(a^*,a^*) = 0$,
it is easy
to observe that $\X= E_1\times E_2$, see also \cite{GMR12}.
Put $K'(\bsx,\bsy) := \sum_{u'\in \mathcal{U}} \gamma_{u'}'\,k_{u'}(\bsx,\bsy)$,
where
$\gamma_{u'}' := \gamma_{u'}$ if $u'\cap v =  u$ and $\gamma_{u'}'=0$ otherwise.
Let the reproducing kernel $J$ be defined by
$J((\bsx_1,\bsx_2), (\bsy_1,\bsy_2)) := K'((\bsx_1;\bsa),(\bsy_2;\bsa))$
for $\bsx_1, \bsy_1 \in E_1$, $\bsx_2, \bsy_2\in E_2$. Then
\begin{equation*}
\begin{split}
J(\bsx,\bsy) &=
\sum_{u'\subset \N\setminus v}
\gamma_{u\cup u'}\, k_{u\cup u'}((\bsx_{v};\bsa),
(\bsy_{v};\bsa))\\
&= \sum_{u'\subset \N\setminus v}
\gamma_{u\cup u'} \prod_{\nu \in u} k(x_\nu,y_\nu)
\prod_{\nu \in u'} k(a_\nu,a_\nu)
= k_{u}(\bsx,\bsy)\, r^2_{v,u,\bsa}.
\end{split}
\end{equation*}
Due to \cite[Lemma~15]{HMGNR10} we thus have
\begin{equation*}
\{\Psi_{v,\bsa}(f) \,|\, f\in B(K') \}
= B(J) = \{f\in H(k_{u}) \,|\, \|f\|_{k_{u}} \le r_{v,u,\bsa} \}.
\end{equation*}
Observe that $f^+_{u,v} \in B(K')$ and that due to (\ref{norm-formel}) the right hand side in (\ref{norm_psi_k}) is invariant
under substituting the norm $\|\cdot\|_K$ by $\|\cdot\|_{K'}$.
Hence we have proved (\ref{norm_psi_k}) and seen that the constant
$r_{v,u,\bsa}$ appearing on the right hand side is optimal. The estimate (\ref{norm_psi_k-1}) follows analogously. Due to Lemma \ref{Lemma5} and (\ref{obs})
we get for $f\in B(K)$
\begin{equation*}
\begin{split}
 \|\Psi_{v,\bsa}(f) \|^2_{K_v} &= \sum_{u\subseteq v; \gamma_u>0} \gamma_u^{-1}
\| \Psi_{v,\bsa}(f^+_{u,v}) \|^2_{k_u} \\
&\le \sum_{u\subseteq v; \gamma_u>0} \gamma^{-1}_u r^2_{v,u,\bsa} \|f^+_{u,v}\|^2_K
\le \max_{u\subseteq v \,;\, \gamma_u>0} \gamma^{-1}_u r^2_{v,u,\bsa}.
\end{split}
\end{equation*}
If $u^* \subseteq v$ satisfies
$\gamma_{u^*}^{-1} r_{v,u^*,\bsa}^2 =  \max_{u\subseteq v \,;\, \gamma_u>0} \gamma^{-1}_u r^2_{v,u,\bsa}$, then we get for $f\in B(K)$ with $f = f^+_{u^*,v}$
\begin{equation}
 \label{evidence}
\| \Psi_{v,\bsa}(f^+_{u^*,v}) \|^2_{K_v}
= \gamma_{u^*}^{-1} \| \Psi_{v,\bsa}(f^+_{u^*,v}) \|^2_{k_{u^*}}
\le \gamma^{-1}_{u*} r^2_{v,u^*,\bsa}.
\end{equation}
Recall that this inequality is invalid for some $f$ with $\|f^+_{u^*,v}\|_K =1$
if we decrease the right hand side of (\ref{evidence}). Thus
$\|\Psi_{v,\bsa}\|_{K\to K_v} = \max_{u\subseteq v\,;\, \gamma_u>0}
\gamma_u^{-1/2}r_{v,u,\bsa}$.
\end{proof}

\begin{remark}
 \label{YES-ANOVA!}
For $u\in\U$  and $f_u\in H_u$ we have
\begin{equation}
\label{anova-int}
 \int_D f_u(\bsx)\, \rho({\rm d}x_j) = 0
\hspace{2ex}\text{for all $j\in u$, $\bsx \in \X$.}
\end{equation}
Indeed, assumption (A~2a) implies that
(\ref{anova-int}) holds for all functions $k_u(\cdot,\bsy)$, $\bsy \in \X$.
Since the linear span of these functions is dense in $H_u$, and since
$I_{\{j\}} \circ \Psi_{\{j\}, \bsx}$ is a continuous linear functional
on $H_u \subseteq H(K)$, identity (\ref{anova-int}) is valid.

With the help of (\ref{anova-int}) it is easy to show that for $v\in \U$
and $f\in H(K_v)$ the uniquely determined decomposition $f=\sum_{u\subseteq v}
f_u$, $f_u\in H_u$,  is exactly the ANOVA decomposition
of $f$ in $L^2(D^v,\rho^v)$. (Similarly as in the proof of Lemma \ref{ANOVA} this can be shown by induction on $|u|$.) In this sense, the uniquely determined decomposition $f=\sum_{u\in\U} f_u$,
$f_u\in H_u$, of $f\in H(K)$ is nothing but the \emph{infinite-dimensional
ANOVA decomposition} of $f$ in $L^2(\X,\mu)$.
\end{remark}
\begin{remark}
\label{interestingQ}
 An interesting question is under what conditions on the weights the
operator norms of the projections $\Psi_{v,\bsa}$ satisfy for some $C>0$
\begin{equation}
 \label{boundedness}
\|\Psi_{v,\bsa}\|_{K\to K_v} \le C
\hspace{2ex}\text{for all $v\in \U$.}
\end{equation}
(This question is in fact relevant for our lower bounds in Theorem
\ref{GenLowBou} and Corollary \ref{lowboundFIW}.) It is easily seen that product weights $\bsgamma$ that satisfy
(\ref{summable}) also satisfy (\ref{boundedness}), see also
\cite[Lemma~7]{HMGNR10}. That this has not necessarily to be the case for general weights, even not
for finite-intersection weights, shows the following example:
Let $a_j = a$ for some $a\in D$ and all $j\in\N$. For a given $\varepsilon >0$ let 
\begin{equation*}
\gamma_{u} =
\begin{cases}
\,j^{-2-\varepsilon}
\hspace{2ex}&\text{if $u = \{j\}$ for some $j\in\N$},\\
\,j^{-1-\varepsilon}
\hspace{2ex}&\text{if $u = \{j,j+1\}$ for some $j\in\N$},\\
\,0
\hspace{2ex} &\text{otherwise.}
\end{cases}
\end{equation*}
The weights
$\bsgamma = (\gamma_u)_{u\in\U}$ we obtain in this way are summable
finite-intersection weights. If $j = \max v$, then (\ref{op_norm_pro})
implies
\begin{equation*}
 \|\Psi_{v,\bsa}\|^2_{K\to K_v} \ge \gamma^{-1}_{\{j\}} r^2_{v, \{j\},\bsa}
\ge \frac{\gamma_{\{j,j+1\}}}{\gamma_{\{j\}}} k(a,a) = j\,k(a,a)
\to \infty
\hspace{2ex}\text{as $j \to \infty$.}
\end{equation*}
For a vector $\bsa$ with identical entries $a\in D$ and finite-intersection weights $\bsgamma$ of order
$\omega$ and with intersection degree $\rho$ the monotonicity condition
\begin{equation}
 \label{monoton}
\gamma_u \ge \gamma_v \hspace{2ex}\text{for all $u,v\in \U$ with $u\subseteq v$ and
$\gamma_u >0$}
\end{equation}
is sufficient to ensure that (\ref{boundedness}) holds, since then we have
for $\emptyset \neq u\subseteq v$, $\gamma_u>0$
\begin{equation*}
  \gamma_u^{-1} r^2_{v,u,\bsa}
\le \sum_{{w\in\N\setminus v \atop \gamma_{u\cup w} >0} } k(a,a)^{|w|}
\le (1+\rho) \max\{1, k(a,a)^{\omega}\},
\end{equation*}
and (\ref{boundedness}) follows from (\ref{op_norm_pro}) and (\ref{summable}).
\end{remark}

\begin{lemma}
\label{Psi-Anova}
For $v,w\in\mathcal{U}$ with $v\subset w$ we have for all $f\in H(K)$
\begin{equation*}
(\Psi_{w,\bsa} - \Psi_{v,\bsa})f
=
\sum_{u\subseteq v}\, \sum_{\emptyset\neq u' \subseteq w\setminus v}
\Psi_{w,\bsa}(f^{+}_{u\cup u',w})
- \sum_{u\subseteq v} \Psi_{v,\bsa}(f^{-}_{u,v,w}).
\end{equation*}
\end{lemma}

\begin{proof}
Let $u\in \mathcal{U}$ satisfy $u\cap v = u\cap w$. Then, due to Lemma \ref{restrict},  we have $\Psi_{w,\bsa}f_u(\bsx) = \Psi_{v,\bsa}f_u(\bsx)$
for all $\bsx \in\X$.
Thus
\begin{equation*}
\begin{split}
(\Psi_{w,\bsa}- \Psi_{v,\bsa})f &= (\Psi_{w,\bsa}- \Psi_{v,\bsa})
\sum_{u\subseteq v}\, \sum_{\emptyset\neq u' \subseteq w\setminus v}
\,\sum_{u''\subset \N\setminus w}
f_{u\cup u'\cup u''}\\
&= \sum_{u\subseteq v}\, \sum_{\emptyset\neq u' \subseteq w\setminus v}
\Psi_{w,\bsa}(f^{+}_{u\cup u',w})
- \sum_{u\subseteq v} \Psi_{v,\bsa}(f^{-}_{u,v,w}).
\end{split}
\end{equation*}
\end{proof}

\begin{lemma}
\label{Bias-Estimate}
For any $v\in \mathcal{U}$ we have
\begin{equation*}
{\rm b}^2_{v,\bsa} := \sup_{f\in B(K)} | I(f) - I(\Psi_{v,\bsa}f)|^2 = \sum_{\emptyset\neq u \subset \N\setminus v}
\gamma_u k_u(\bsa,\bsa).
\end{equation*}
\end{lemma}

\begin{proof}
Let $h_{v,\bsa}$ denote the representer of $I\circ \Psi_{v,\bsa}$ in $H(K)$, i.e.,
\begin{equation*}
h_{v,\bsa}(\bsx) = \int_{D^v} K(\bsx,(\bsy_v;\bsa))\, \rho^v({\rm d}\bsy_v)
=  \sum_{u\in\mathcal{U}} \gamma_u \int_{D^v} k_u(\bsx,(\bsy_v;\bsa))\,
\rho^v({\rm d}\bsy_v).
\end{equation*}
Due to (\ref{anova-int}) the last integral is zero if $v\cap u \neq \emptyset$.
Hence
\begin{equation*}
h_{v,\bsa}(\bsx)
=  \sum_{u \subset \N\setminus v} \gamma_u k_u(\bsx, \bsa).
\end{equation*}
Due to (\ref{h=1}) we have
\begin{equation*}
\begin{split}
{\rm b}^2_{v,\bsa} = &\sup_{f\in B(K)} | I(f) - I(\Psi_{v,\bsa}f)|^2  = \|h-h_{v,\bsa}\|^2_K
= \bigg\| \sum_{\emptyset \neq u \subset \N\setminus v}\gamma_u k_u(\cdot, \bsa) \bigg\|^2_K\\
= &\sum_{\emptyset \neq u \subset \N\setminus v}\gamma_u\left\|k_u(\cdot, \bsa) \right\|^2_{k_u}
= \sum_{\emptyset \neq u \subset \N\setminus v}\gamma_u k_u(\bsa, \bsa).
\end{split}
\end{equation*}
\end{proof}

\subsection{Cost and error}

In this subsetion we present the cost models introduced in \cite{CDMR09}.
Apart from slight generalizations, we essentially follow
the representation in \cite[Sect.~3]{HMGNR10}.

For $\bsx\in D^\N\setminus \X$ we put $f(\bsx) =0$ for all $f\in H(K)$. Let $\$(\nu)$, $\nu\in\N\cup\{0\}$, be a monotone increasing cost function.
Here we will usually assume that $\$(\nu) = O(\nu^s)$ (for upper error
bounds) or $\$(\nu) = \Omega(\nu^s)$ (for lower error bounds), where $s>0$.
(Corresponding results for the case $s=0$ can easily be obtained by taking the limit $s\to 0$; anyhow, we believe that the most interesting case is $s\ge 1$.)

In the \emph{fixed subspace sampling model} function evaluations are
only possible in points from a finite-dimensional affine subspace
\begin{equation*}
 \X_{v,\bsa} := \{ \bsx \in D^\N \,|\, x_j = a_j \hspace{1ex}\text{for all}
\hspace{1ex} j\in\N\setminus v\}
\end{equation*}
of $\X$ for a given $v\in\U$ and an \emph{admissable anchor} $\bsa\in\X$, and the cost for each function
evaluation is given by a cost function
\begin{equation}
\label{fixcost}
 c_{v,\bsa}(\bsx) :=\begin{cases}
\, \$(|v|)
\hspace{2ex}&\text{if $\bsx \in \X_{v,\bsa}$},\\
\, \infty
\hspace{2ex} &\text{otherwise.}
\end{cases}
\end{equation}
In the \emph{variable subspace sampling model}\footnote{To distinguish this cost model
clearly from the more generous one defined in \cite{KSWW10a} it seems to be more accurate to
rename it
``nested subspace sampling model'' as done in \cite{DG12, Gne12}; since here we do not consider the cost model from \cite{KSWW10a}, we stay with the original name.} function evaluations can
be done in a sequence of affine subspaces
\begin{equation*}
 \X_{v_1,\bsa} \subset \X_{v_2,\bsa} \subset \cdots
\end{equation*}
for a strictly increasing sequence $\seqv = (v_i)_{i\in\N}$ of sets
$\emptyset \neq v_i \in \U$ and an admissable anchor $\bsa\in \X$,
and the cost for each function
evaluation is given by the cost function
\begin{equation}
\label{varcost}
 c_{\seqv,\bsa}(\bsx) := \inf\{ \$(|v_i|) \,|\, \bsx\in \X_{v_i,\bsa} \},
\end{equation}
where we use the standard convention that $\inf \emptyset = \infty$.
Let $C_{\fix}$ and $C_{\var}$ denote the set of all cost functions of the
form (\ref{fixcost}) and (\ref{varcost}), respectively.

In general we assume that all $\bsa\in\X$ are admissable anchors, but in
some situation we restrict ourselves
to admissable anchors of the form $\bsa=(a,a,\ldots)\in \X$, $a\in D$, as
done in \cite{HMGNR10}.

We consider randomized algorithms for integration of functions $f\in H(K)$
and, as in \cite{HMGNR10}, refer for a formal definition to \cite{CDMR09, N88, TWW88}.
The cost of an algorithm is defined to be the sum of the cost of all function
evaluations. For a randomized algorithm $Q$ the cost is a random variable,
which may depend on the function $f$. That is why we denote this random
variable by $\cost_c(Q,f)$, where $c$ denotes the relevant cost function from $C_{\fix}$
or $C_{\var}$.

The \emph{worst case cost} of a randomized algorithm $Q$ on a class of integrands $F$ is given
by
\begin{equation*}
 \cost_{\fix}(Q,F) := \inf_{c\in C_{\fix}} \sup_{f\in F}
\EE(\cost_c(Q,f))
\end{equation*}
in the fixed subspace sampling model and by
\begin{equation*}
 \cost_{\var}(Q,F) := \inf_{c\in C_{\var}} \sup_{f\in F}
\EE(\cost_c(Q,f))
\end{equation*}
in the variable subspace sampling model.

The \emph{randomized error} $e(Q,F)$ of approximating the integration functional $I$ by $Q$ on $F$ is defined as
\begin{displaymath}
e(Q,F) := \bigg(\sup_{f \in F} \E \left( \left( I(f) - Q(f) \right)^2 \right) \bigg)^{1/2}\, .
\end{displaymath}
For $N\in\R$ let us define the \emph{$N$th minimal errors} by
\begin{equation*}
 e_{N,\fix}(F) := \inf\{ e(Q,F) \,|\, \cost_{\fix}(Q,F) \le N\}
\end{equation*}
and
\begin{equation*}
 e_{N,\var}(F) := \inf\{ e(Q,F) \,|\, \cost_{\var}(Q,F) \le N\}.
\end{equation*}

\section{Lower bounds}
\label{LOWBOU}

For a fixed anchor $\bsa\in\X$ and a sequence of weights
$(\gamma_u)_{u\in\mathcal{U}}$ satisfying (\ref{summable}) let
$u_1,u_2,\ldots$ be an ordering of the non-empty sets $u\in\mathcal{U}$
with $\gamma_u >0$ for which $\widehat{\gamma}_{u_1} \ge \widehat{\gamma}_{u_2}
\ge\cdots$ holds, where $\widehat{\gamma}_u := \gamma_u \, k_u (\bsa,\bsa)$.
Let $u_0:=\emptyset$. Furthermore, we put
\begin{equation*}
\decay_{\bsgamma} :=
\sup \left\{ p\in \R \,\Big|\, \lim_{j\to\infty}
\widehat{\gamma}_{u_j}j^p =0
\right\}.
\end{equation*}

\subsection{General weights}

The next two lemmas are helpful for establishing lower bounds
for the randomized error of numerical integration.

\begin{lemma}
\label{lemlowbound1} Let $\theta\in (1/2,1]$, $v\in\U$, and  let $Q$ be a randomized algorithm that satisfies
 $\PP\big( Q(f) = Q(\Psi_{v,\bsa}f) \big) \ge \theta$
for all $f\in B(K)$.
Then
\begin{displaymath}
e(Q,B(K)) \ge \max\left\{ \frac{\sqrt{2\theta -1}\, {\rm b}_{v,\bsa}}{1+\|\Psi_{v,\bsa}\|_{K\to K_v}} \,, e(Q,B(K_{v})) \right\} \, .
\end{displaymath}
\end{lemma}

\begin{proof}
The proof adapts the proof idea from \cite[Lemma~8]{HMGNR10}.
Put
$\hat{r} := \| \Psi_{v,\bsa} \|_{K\to K_v}$.
Then we have for $f \in B(K)$,
\begin{displaymath}
g= \frac{f - \Psi_{v,\bsa} (f)}{1+\hat{r}} \in B(K) \, .
\end{displaymath}
Furthermore, we have $\Psi_{v,\bsa}(g) = \Psi_{v,\bsa}(-g) =0$.
Let $A$ denote the event $\{Q(g) = Q(-g) \}$. Then $\PP(A) \ge 2\theta -1$.
Hence
\begin{equation*}
 \begin{split}
&e(Q,B(K))^2 \geq
\max \left\{ \EE \left( \left( I(g) - Q( g) \right)^2 \right) ,
\EE \left( \left(  I(-g) - Q(-g)  \right)^2 \right) \right\} \\
\geq
&\max \left\{ \int_A \left( I(g) - Q( g) \right)^2
\,\PP({\rm d}\omega),
\int_A \left(  I(-g) - Q(-g) ) \right)^2  \,\PP({\rm d}\omega) \right\}\\
\geq &(2\theta -1) \vert I(g) \vert^2 = (2\theta-1)(1+\hat{r})^{-2} \vert I(f) - I(\Psi_{v,\bsa}(f)) \vert^2 \, .
\end{split}
\end{equation*}
Since $B \left( K_{v} \right) \subseteq B(K)$, we have additionally
$e \left( Q, B(K) \right) \geq e \left( Q, B(K_{v}) \right)$.
\end{proof}

We provide now a general lower bound for the randomized error
of arbitrary randomized algorithms and arbitrary weights.
\begin{theorem}
 \label{GenLowBou}
Assume that $\$(\nu) = \Omega(\nu^s)$ for some $s>0$ and that there exists a
$p>1$ and a $\sigma>0$ such that
\begin{equation*}
 \frac{{\rm b}_{v,\bsa}}{1+\|\Psi_{v,\bsa}\|_{K\to K_v}}
= \Omega \left( |v|^{\frac{ \sigma (1-p) }{2}} \right)
\hspace{2ex}\text{for all $v\in\U$.}
\end{equation*}
Assume further that $\gamma_{\{1\}} >0$ and that there exists
an $\alpha >0$ with $e_N(B(K_{\{1\}}))^2 = \Omega(N^{-\alpha})$.
Then we have for fixed subspace sampling
\begin{equation*}
e_{N,\fix}(B(K))^2 = \Omega \left( N^{-\frac{\alpha \sigma (p-1)}{\alpha s+ \sigma (p -1) }} \right),
\end{equation*}
and for variable subspace sampling
\begin{equation*}
e_{N,\var}(B(K))^2 = \Omega \left( N^{-\min \left\{ \alpha, \frac{ \sigma ( p-1 ) }{s} \right\}} \right).
\end{equation*}
\end{theorem}

\begin{proof}
Let $Q$ be a randomized algorithm.
In the fixed subspace sampling regime our proof is a slight modification
of the proof of \cite[Thm.~2]{HMGNR10}:
If $\cost_{\fix}(Q,B(K)) \le N$, then
there exists a set $v\in\U$ and an anchor $\bsa\in \X$ such that
$\EE(\cost_{c_{v,\bsa}}(Q,f)) \le N+1$ for every $f\in B(K)$.
This implies for every $f\in B(K)$ that
$\PP( Q(f) = Q(\Psi_{v,\bsa}\,f)) =1$. Due to Lemma \ref{lemlowbound1}
we get $e(Q,B(K))^2 = \Omega(|v|^{\sigma (1-p)})$.

The expected number of evaluations
of $Q$ is at most of order $O(N/|v|^s)$. Thus we have
\begin{equation*}
 e(Q,B(K))^2 \ge e(Q,B(K_{\{1\}}))^2 = \Omega \left(\left( \frac{N}{|v|^s}
\right)^{-\alpha} \right).
\end{equation*}
Now it is easily verified that
\begin{equation*}
 \left( \frac{N}{|v|^s} \right)^{-\alpha} + |v|^{\sigma (1-p) }
= \Omega \left( N^{-\frac{\alpha \sigma (p-1)}{\alpha s+ \sigma (p -1 ) }} \right).
\end{equation*}
Let us turn to the variable subspace sampling regime: If $\cost_{\var}(Q,B(K)) \le N$, then there exists an increasing sequence $\seqv = (v_i)_{i\in\N}$,
$\emptyset \neq v_i\in \U$, and an anchor $\bsa \in \X$ such that
$\EE(\cost_{c_{\seqv,\bsa}}(Q,f)) \le N+1$ for every $f\in B(K)$. Let $m$ be
the largest integer satisfying $\$(|v_m|) \le 4(N+1)$. That implies for all
$f\in B(K)$ that $\PP( Q(f) = Q(\Psi_{v_m,\bsa}\,f)) \ge 3/4$. Due to Lemma \ref{lemlowbound1} we get $e(Q,B(K))^2 = \Omega(|v_m|^{ \sigma (1-p ) })$.
Since $4(N+1) = \Omega(|v_m|^s)$, we obtain
$e(Q,B(K))^2 = \Omega(N^{\frac{ \sigma (1-p)  }{s}})$. Furthermore, we have
\begin{equation*}
 e(Q,B(K))^2 \ge e(Q,B(K_{\{1\}}))^2 = \Omega(N^{-\alpha}).
\end{equation*}
This concludes the proof.
\end{proof}

\subsection{Finite-intersection and product weights}

As already discussed in Remark \ref{interestingQ}, for product weights
the operator norm
$\|\Psi_{v,\bsa}\|_{K\to K_v}$ is uniformly bounded in $v\in\U$,
and the same holds true for finite-intersection weights $\bsgamma$
as long as the anchor $\bsa\in\X$ has identical entries $a\in D$ and the
monotonicity condition (\ref{monoton}) is satisfied.
\begin{lemma}
\label{lemlowboundAva}
Let $\bsgamma$ be finite-intersection
or product weights, and let $p>\decay_{\bsgamma}$. Let $v\in\U$. Then we have
${\rm b}^2_{v,\bsa} = \Omega(\vert v \vert^{1-p})$.
\end{lemma}

\begin{proof}
Let $\bsgamma$ be finite-intersection weights.
Let $\eta$ be as in condition (\ref{cond}).
Note that the set $\left\{ i| u_i \cap v \neq \emptyset \right\}$ contains at most $\eta \vert v \vert$ elements. Hence
\begin{equation*}
 \begin{split}
{\rm b}^2_{v,\bsa} &= \sum_{\emptyset \neq u \subset\nn \setminus v} \gamma_{u}
k_u(\bsa,\bsa) = \sum_{\stackrel{j\in\N:}{\emptyset \neq u_j \subset \nn \setminus v}} \gamma_{u_j} k_{u_j}(\bsa,\bsa) = \sum_{\stackrel{j\in\N:}{j \not\in \left\{  i | u_i \cap v \neq \emptyset \right\}}} \gamma_{u_j} k_{u_j}(\bsa,\bsa) \\
&\ge \sum^{\infty}_{j= \vert v \vert \eta +1} \gamma_{u_j} k_{u_j}(\bsa,\bsa)
\geq \sum^{2 \vert v \vert \eta }_{j= \vert v \vert \eta +1} \gamma_{u_j} k_{u_j}(\bsa,\bsa) \geq \vert v \vert \eta \widehat{\gamma}_{u_{2 \vert v \vert \eta}} = \Omega \left( \vert v \vert^{1-p} \right) \, .
 \end{split}
\end{equation*}
In the case of product weights, ${\rm b}^2_{v,\bsa} = \Omega(\vert v \vert^{1-p})$ was proved in \cite[p.~243]{HMGNR10}.
\end{proof}

\begin{remark}
 The statement of Lemma \ref{lemlowboundAva} does not hold for
arbitrary weights, as shown by the following example: Consider
weights $(\gamma_u)_{u\in\U}$ defined by $\gamma_u >0$ if
$u=[d]$ for some $d\in \N$ and $\gamma_u=0$ otherwise.
Then
\begin{equation*}
 {\rm b}^2_{v,\bsa} = \sum_{\emptyset \neq u \subset \N\setminus v}
\gamma_u k_u(\bsa,\bsa) = 0
\hspace{2ex}\text{for all $1\in v\in \U$.}
\end{equation*}
\end{remark}

Remark \ref{interestingQ}, Lemma \ref{lemlowboundAva}, and Theorem \ref{GenLowBou} lead directly to the following lower bounds for finite-intersection and
product weights. Note that for $s=1$ the lower bound for product weights in
the fixed subspace sampling model has already been proved in \cite[Thm.~2]{HMGNR10}.

\begin{corollary} \label{lowboundFIW}
Let $\bsgamma$ be finite-intersection weights  or
product weights. In the case of finite-intersection weights we additionally
assume that the weights satisfy the monotonicity
condition (\ref{monoton}) and  $\gamma_{\{1\}} >0$, and that all admissable anchors $\bf{a} \in\X$ for fixed or
variable subspace sampling are of the form
$\bsa = (a,a,\ldots)$ for some suitable $a\in D$. Let
$\$(\nu) = \Omega(\nu^s)$ for some $s>0$, and let
$p> \decay_{\bsgamma}$.
Assume further that there exists
an $\alpha >0$ with $e_N(B(K_{\{1\}}))^2 = \Omega(N^{-\alpha})$. Then
we have for fixed subspace sampling
\begin{displaymath}
e_{N,\fix}(B(K))^2 =
\Omega \left( N^{-\frac{\alpha(p-1)}{\alpha s +p-1}} \right) \,,
\end{displaymath}
and for variable subspace sampling
\begin{displaymath}
e_{N,\var}(B(K))^2 = \Omega \left( N^{-\min \left\{ \alpha, \frac{p-1}{s} \right\}} \right) \,.
\end{displaymath}
\end{corollary}

\section{Multilevel Algorithm} \label{secmultilevelalg}

In this section, we discuss multilevel algorithms, firstly in generality mostly relying on \cite{G10}, and subsequently show how to tailor them to finite-intersection and product weights.

\subsection{General weights} \label{subsecgenweights}

Let us describe the general form of the multilevel algorithms we want to use more precisely:
Let $L_0:=0$, let $L_1<L_2<L_3 <\ldots$ be natural numbers,
and let
\begin{equation}
\label{vk12}
v^{(1)}_k := \cup_{j\in [L_k]}^{} u_j
\hspace{2ex}\text{and}\hspace{2ex}
v^{(2)}_k := [L_k]
\hspace{2ex}\text{for $k\in\N$}.
\end{equation}
For general weights we will use the sets
$v^{(1)}_k$, $k=1,2,\ldots$.
In special cases as, e.g., for product weights or the lexicographically-ordered
weights defined in \cite{G10}, it is more convenient to
make use of the special ordering of the corresponding set system $u_j$, $j\in\N$,
and choose
the sets $v^{(2)}_k$ for $k=1,2,\ldots$. In all definitions and results that hold for both choices
of the $v_k^{(i)}$, $i=1,2$, we simply write $v_k$. Put $v_0:=\emptyset$.
We will choose the numbers $L_1,L_2,\ldots$ in general such that
$|v_k| = \Theta(b^{k})$ for some $b\in (1,\infty)$.
(Here a default choice would be $L_k = 2^{k-1}$.) Let
\begin{equation*}
V_1 := \{j\in\N\,|\, u_j\subseteq v_1\}
\end{equation*}
and
\begin{equation*}
V_k := \{j\in\N \,|\,  u_j\subseteq v_k
\hspace{1ex}\text{and}\hspace{1ex} u_j\not\subseteq v_{k-1}\}
\hspace{2ex}\text{for $k\ge 2$.}
\end{equation*}
Let us furthermore define
\begin{equation*}
U(m) := \cup_{k=1}^m V_k \cup \{0\}
\hspace{2ex}\text{and}\hspace{2ex}
\tail_{\bsgamma} (m):= \sum_{j\notin U(m)}^\infty
\widehat{\gamma}_{u_j},
\end{equation*}
where the weights $\widehat{\gamma}_{u_j}$ are defined as in Subsection
\ref{LOWBOU}.
Let us fix an anchor $\bsa\in\X$.
We use the short hands $\Psi_0 := 0$,
\begin{equation*}
\Psi_k := \Psi_{v_k,\bsa}
\hspace{2ex}\text{and}\hspace{2ex}
{\rm b}^2_{k} := {\rm b}^2_{v_k,\bsa},
\hspace{2ex}\text{for $k=1,2,\ldots$,}
\end{equation*}
as well as
$$
f^+_{u,k} := f^+_{u,v_k}
\hspace{2ex}\text{and}\hspace{2ex}
f^-_{u,k} := f^-_{u,v_{k-1},v_k}.
$$
Furthermore, let
\begin{equation*}
 r_{k,u}:= r_{v_k,u,\bsa}
\hspace{2ex}\text{and}\hspace{2ex}
\tilde{r}_{k,u} := \tilde{r}_{v_k,v_{k-1},u,\bsa}.
\end{equation*}
For natural numbers $n_1 \ge n_2 \ge \cdots \ge n_m$, we consider randomized
algorithms $Q_{v_k,\bsa}$ of the form
\begin{equation}
\label{algobaustein}
Q_{v_k,\bsa}(f)
:= \sum_{j=1}^{n_k} w_{j,k}
f(\bst^{(j,k)}_{v_{k}}; \bsa),
\hspace{2ex} w_{j,k}\in \R,\,\bst^{(j,k)}_{v_{k}} \in D^{v_k},
\end{equation}
that satisfy (\ref{summe=1}) and condition (*) of Lemma \ref{ANOVA} for
$\mathcal{V} = \{u\subseteq v_k \,|\, \gamma_u >0\}$.
We use additionally the shorthand
\begin{equation}
Q_k(f) := Q_{v_k,\bsa}(f - \Psi_{k-1}f).
\end{equation}
Define the \emph{randomized multilevel algorithm} $Q$ via
\begin{equation}
\label{multilevel-algo}
Q(f)
:= \sum_{k=1}^m Q_k(f)
= \sum_{k=1}^m \sum_{j=1}^{n_k} w_{j,k}
(f - \Psi_{k-1}f) (\bst^{(j,k)}_{v_{k}};\bsa),
\end{equation}
where the random variables $Q_k(f)$, $k=1,\ldots, m$, are supposed to be independent.

Since $\E(Q_{v_k,\bsa}(f)) = I(\Psi_kf)$ for all
$k$, see Remark \ref{Unbiased}, we have $\E(Q(f)) = I(\Psi_m(f))$. Thus
$Q(f)$ is an \emph{unbiased estimator} of $I(\Psi_m(f))$, and we obtain
\begin{equation}
\label{rand-case-error}
\E(I(f) - Q(f))^2 = (I(f) - I(\Psi_m f))^2 +\Var(Q(f)).
\end{equation}
Since $Q_k(f)$, $k=1,\ldots,m$, are independent random variables,
we have the following identity for the variance of $Q(f)$:
\begin{equation}
\label{variation}
\Var(Q(f)) = \sum^m_{k=1} \Var(Q_{k}(f)).
\end{equation}

\begin{lemma}
\label{Var-ANOVA}
For all $f\in H(K)$ and all $k \in [m]$ we have
\begin{equation}
\label{varidentity}
\Var(Q_k(f)) =   \sum_{j\in V_k}
\Var ( Q_{v_k,\bsa} ( \Psi_k (
f^+_{u_j,k} ) ) ) + \sum_{\emptyset \neq u\subseteq v_{k-1}}
\Var ( Q_{v_k,\bsa} ( \Psi_{k-1} (
f^-_{u,k} ) ) ) .
\end{equation}
\end{lemma}

\begin{proof}
For $k \in [m]$ we obtain from Lemma \ref{Psi-Anova}
\begin{equation*}
\Var(Q_k(f))
=  \Var \left( \sum_{j\in V_{k}}
Q_{v_k,\bsa}(\Psi_k(f^+_{u_j,k}))
+ \sum_{u\subseteq v_{k-1}} Q_{v_k,\bsa}(-\Psi_{k-1}(f^-_{u,k}))\right).
\end{equation*}
Recall from Lemma \ref{Analogon-Lemma7} that $\Psi_k(f^+_{u_j,k}) \in H_{u_j}$
and $-\Psi_{k-1}(f^-_{u,k}) \in H_u$.
Since $Q_{v_k,\bsa}$ satisfies condition (*)
in Lemma \ref{ANOVA} for $\mathcal{V} = \{u\subseteq v_k \,|\, \gamma_u >0\}$,
we obtain with Remark \ref{YES-ANOVA!} and
Lemma \ref{ANOVA}
\begin{equation*}
\Var(Q_k(f))
=  \sum_{j\in V_{k}}
\Var(Q_{v_k,\bsa}(\Psi_k(f^+_{u_j,k})) )
+ \sum_{u\subseteq v_{k-1}} \Var(Q_{v_k,\bsa}(\Psi_{k-1}(f^-_{u,k}))) .
\end{equation*}
The statement of Lemma \ref{Var-ANOVA} follows after
observing that for $u=\emptyset$ the function $\Psi_{k-1}(f^-_{\emptyset,k})
\in H_\emptyset$ is constant, and therefore also
$Q_{v_k,\bsa}(\Psi_{k-1}(f^-_{\emptyset,k}))$.
\end{proof}

Due to (\ref{rand-case-error}) and (\ref{variation}), we get for $f\in B(K)$
\begin{equation}
 \label{zeile1}
\EE(I(f) - Q(f))^2
\le \sum^m_{k=1} \Var(Q_{k}(f)) + {\rm b}^2_{m}
\end{equation}
and for $k \in [m]$ we have that \eqref{varidentity} holds. Assume now that there exist for every $k\in\N$ algorithms of the form
(\ref{algobaustein}) that satisfy (\ref{summe=1}) and condition (*) of Lemma \ref{ANOVA} for $\mathcal{V} = \{u\subseteq v_k \,|\, \gamma_u >0\}$,
and for which there exists a $\tau>0$
and for each $u\subseteq v_k$ with $\gamma_u>0$
a constant $C_{u,k,\tau}$ such that
\begin{equation}
 \label{algo-annahme}
\Var(Q_{v_k,\bsa}(f_u)) \le C_{u,k,\tau}(n_k - 1)^{-\tau} \|f_u\|^2_{k_u}
\hspace{2ex}\text{for all $f_u\in H_u$.}
\end{equation}
Then with Lemma \ref{Analogon-Lemma7} and Remark \ref{Obs}, we obtain for all $f\in B(K)$
\begin{equation*}
 \begin{split}
  &\Var(Q_k(f))\\
 \le &\Bigg( \sum_{j\in V_k} C_{u_j,k,\tau}
\|\Psi_k(f_{u_j,k}^+)\|^2_{k_{u_j}}
+ \sum_{\emptyset \neq u \subseteq v_{k-1}} C_{u,k,\tau}
\|\Psi_{k-1}(f_{u,k}^-)\|^2_{k_{u}} \Bigg) (n_k - 1)^{-\tau}\\
\le &\left( \max_{j\in V_k} C_{u_j,k,\tau}  r^2_{k,u_j}
+ \max_{\emptyset \neq u \subseteq v_{k-1}} C_{u,k,\tau}  \tilde{r}^2_{k,u}
\right) (n_k - 1)^{-\tau}.
 \end{split}
\end{equation*}
Hence
\begin{equation} \label{Fehlerabschaetzung}
e(Q,B(K))^2 \le \left( \sum^m_{k=1} \left( \max_{j\in V_k} C_{u_j,k,\tau} r^2_{k,u_j}
+ \max_{\emptyset \neq u \subseteq v_{k-1}} C_{u,k,\tau}  \tilde{r}^2_{k,u}
\right) (n_k-1)^{-\tau} + {\rm b}^2_m \right).
\end{equation}
The aim is now to minimize the right hand side of this error bound for given
cost by choosing $m$, $L_1,\ldots,L_m$,  and $n_1,\ldots,n_m$ essentially optimal.
To do so, one needs more specific information about the
constants $C_{u,k,\tau}$ and about the weights $(\gamma_u)_{u\in\U}$.

\subsection{Finite-Intersection Weights}

Let $(\gamma_{u_i})_{i\in \N}$ be finite-intersection weights of
finite order $\omega$. Let $\eta\in\N$ be such that the set system
$u_j$, $j=1,2,\ldots$, satisfies (\ref{cond}). Put
\begin{equation} \label{eqdforfiniteintersectionweights}
d:= \eta(\omega-1)+1.
\end{equation}
Here we choose the sets $v_k = v_k^{(1)} = \cup_{j\in [L_k]} u_j$ in (\ref{vk12}),
where the numbers $L_1,L_2,\ldots$ will be determined later.
Observe that $\eta^{-1} L_k \le |v_k| \le \omega L_k$ for all $k\in\N$.
We assume that $\bsa = (a,a,\ldots)$, where $a\in D$ satisfies (\ref{anker_a}).
Notice that this assumption leads to
\begin{equation*}
0<\min\{1, k(a,a)^\omega\} \le k_{u_j}(\bsa,\bsa) \le \max\{1, k(a,a)^\omega\}
\hspace{2ex}\text{for all $j\in\N$.}
\end{equation*}

\begin{proposition}
\label{Prop3.11}
Assume that for $d$ given by \eqref{eqdforfiniteintersectionweights} and all $n\in\N$ there exist randomized linear algorithms
$Q_n=Q_{[d],n}$ of the form (\ref{alg-form}), that satisfy condition (*) of Lemma \ref{ANOVA} for $\mathcal{V} = \{u\subseteq [d]\}$, and for
which there exist constants $\alpha, \beta$, and $C_d$, independent of $n$, such that
\begin{equation} \label{eqconditionquadrulefinintersecweights}
 \Var(Q_{n} (f_u)) \le C_{d} n^{-\alpha} \ln(n)^\beta \|f_u\|^2_{k_u}
\hspace{2ex}\text{for all $u\subseteq [d]$, $f_u\in H_u$.}
\end{equation}
Then we find for all $k\in\N$ and all $n_k\in \N$ randomized linear algorithms
$Q_{v_k,\bsa}$ of the form (\ref{algobaustein}), that satisfy condition (*) of Lemma \ref{ANOVA} for $\mathcal{V} = \{u\subseteq v_k \,|\, \gamma_u >0\}$
and
\begin{equation}
\label{var-est}
 \Var(Q_{v_k,\bsa}(f_{u_j})) \le C_{d} n^{-\alpha}_k \ln(n_k)^\beta
 \|f_{u_j}\|^2_{k_{u_j}}
\hspace{2ex}\text{for all $u_j\subseteq v_k$, $f_{u_j}\in H_{u_j}$.}
\end{equation}
If all $Q_n$ satisfy (\ref{summe=1}), then the $Q_{v_k,\bsa}$ satisfy
(\ref{summe=1}), too.
\end{proposition}

\begin{proof}
Consider for given $n_k$ the algorithm $Q_{n_k}$.
Due to Lemma \ref{phi} we find a mapping $\phi: \N \to [d]$ such that for
all $j\in\N$ the restriction $\phi|_{u_j}$ is injective.
We obtain the random point
$\bst^{(i,k)}_{v_k} \in D^{v_k}$
by defining its $\nu$th component by
\begin{equation}
\label{update}
t^{(i,k)}_{v_k,\nu}(\omega):= x^{(i)}_{\phi(\nu)}(\omega_{\phi(\nu)})
\hspace{2ex}\text{for all $\nu\in v_k$, $\omega\in \Omega^{[d]}$, $i=1,\ldots, n_k$.}
\end{equation}
Notice that the projection of $\bst^{(i,k)}_{v_k}$ to $[0,1]^{u_j}$ for
$u_j\subseteq v_k$ consists of $|u_j|$ different components of $\bsx^{(i)}$.
We choose the coefficients of $Q_{v_k,\bsa}$ to be the coefficients of $Q_{n_k}$,
i.e., $w_{j,k} := w_j$.
Observe that the resulting randomized algorithm  $Q_{v_k,\bsa}$ satisfies
condition (*) of Lemma \ref{ANOVA} for $\mathcal{V} = \{u\subseteq v_k \,|\, \gamma_u >0\}$.
It is easily seen that (\ref{var-est}) holds.
\end{proof}

\subsubsection{Variable Subspace Sampling}

In the variable subspace sampling cost model we have the following result on the multilevel algorithm and finite-intersection weights.

\begin{theorem}
\label{Error-FIW}
Assume that for $d$ given by \eqref{eqdforfiniteintersectionweights} and all $n\in\N$ there exist an algorithm $Q_n=Q_{[d],n}$
as in the condition of Proposition \ref{Prop3.11} that satisfies (\ref{summe=1}). Consider the multilevel algorithm $Q$ defined in (\ref{multilevel-algo}), where
the algorithms $Q_{v_k,\bsa}$ are as in Proposition \ref{Prop3.11}.
Let $N$ be the cost of the algorithm corresponding to the cost function
$\$(\nu) = O(\nu^s)$ for some $s>0$. Then there exists for all $\delta >0$ a constant
$C=C(d, \delta)$ such that
\begin{equation*}
e(Q,B(K))^2 \le C N^{- \alpha+\delta}
\hspace{2ex}\text{if $ \decay_{\bsgamma} \ge 1+\alpha s$,}
\end{equation*}
and
\begin{equation*}
e(Q, B(K))^2 \le C N^{-\frac{\decay_{\bsgamma}-1}{s} +\delta}
\hspace{2ex}\text{if $ 1 + \alpha s > \decay_{\bsgamma} > 1$.}
\end{equation*}
\end{theorem}

In the case, where the assumptions of Corollary \ref{lowboundFIW} hold for the same $\alpha$ as in Theorem \ref{Error-FIW}, our lower bound on $e_{N, \var}(B(K))$ shows
that the upper bounds in Theorem \ref{Error-FIW} are essentially sharp.

\begin{proof}
Let $f\in B(K)$.
Due to Proposition \ref{Prop3.11}, we find for every $\delta >0$ a constant
$C= C_{d,\delta}$ with
\begin{equation*}
\Var(Q_{v_k,\bsa}(\Psi_k (f^+_{u_j,k}))) \le C n_k^{-(\alpha-\delta)} \|\Psi_k(f^+_{u_j,k})\|^2_{k_{u_j}}
\end{equation*}
for all $j\in V_k$, and
\begin{equation*}
\Var(Q_{v_k,\bsa}(\Psi_{k-1} (f^-_{u,k}))) \le C n_k^{-(\alpha-\delta)} \|\Psi_{k-1}(f^-_{u,k})\|^2_{k_{u}}
\end{equation*}
for all $u\subseteq v_{k-1}$.
Furthermore, Lemma \ref{Analogon-Lemma7} gives us
\begin{equation*}
\|\Psi_k(f^+_{u_j,k})\|^2_{k_{u_j}} \le r^2_{k,u_j} \|f^+_{u_j,k}\|^2_K
\end{equation*}
and
\begin{equation*}
\|\Psi_{k-1}(f^-_{u,k})\|^2_{k_{u}} \le \tilde{r}^2_{k,u} \|f^-_{u,k}\|^2_K.
\end{equation*}
Now
\begin{equation}
\label{est_r}
 r^2_{k,u_j}
= O(\tail_{\bsgamma}(k-1))
\hspace{2ex}\text{and}\hspace{2ex}
 \tilde{r}^2_{k,u}
= O(\tail_{\bsgamma}(k-1)),
\end{equation}
and we set $\sigma_k := \tail_{\bsgamma}(k-1)$ for $k=1,2,\ldots$.
We get from \eqref{varidentity}
\begin{equation*}
\Var(Q_{k}(f)) =
O \Bigg( n_k^{-(\alpha-\delta)} \sigma_k \Bigg( \sum_{j\in V_k} \|f^+_{u_j,k} \|^2_K
+ \sum_{u\subseteq v_{k-1}} \|f^-_{u,k-1}\|^2_K \Bigg) \Bigg),
\end{equation*}
and the two sums in parentheses are bounded by $2\|f\|^2_K \le 2$, see Remark \ref{Obs}. Hence
\begin{equation}
\Var(Q_{k}(f)) =
O \left( n_k^{-(\alpha-\delta)} \sigma_k \right).
\end{equation}
Thus we have, due to (\ref{zeile1}) and since $b^2_m=O(\tail_{\bsgamma}(m))$,
\begin{equation}
\label{hecke}
e(Q,B(K))^2 = O \Bigg( \sum^m_{k=1} \sigma_k n_k^{-(\alpha-\delta)} + \sigma_{m+1} \Bigg).
\end{equation}
Let $S> M$, where $M:= \sum^m_{k=1}L_k^s$,  be given.
We assume that $L_k = L \lceil A^k \rceil$ for  a fixed $L\in\N$ and
$A>1$.
We want to find the minimum
$\bsx^*=(x_1^*,\ldots,x^*_m)$ of the function
$$
G(\bsx)=\sum^m_{k=1}\sigma_k x^{-(\alpha-\delta)}_k
\hspace{2ex}\text{subject to the constraint}\hspace{2ex}
\sum^m_{k=1}x_k L_{k}^s = S.
$$
Due to Lagrange's multiplier theorem there exists a
$\lambda\in\R$ such that
$\grad G(\bsx^*) = \lambda(L_1^s,\ldots, L_m^s)$.
This relation and the constraint imply that the minimum $\bsx^*$
is given by
\begin{equation}
\label{minimum}
x^*_k = \widetilde{C} \sigma_k^{\frac{1}{\alpha+1-\delta}} L_k^{-\frac{s}{\alpha+1-\delta}},
\hspace{2ex}\text{where}\hspace{2ex}
\widetilde{C} = S \left( \sum^m_{k=1} \sigma_k^{\frac{1}{\alpha+1-\delta}}
L_k^{\frac{(\alpha-\delta)s}{\alpha+1-\delta}}
\right)^{-1}.
\end{equation}
For $k=1,2,\ldots,m$ we choose $n_k:= \lceil x^*_k \rceil $.
This leads to
\begin{equation*}
N
= O\left( \sum^m_{k=1} n_k L^s_k \right)
= O(S +M)=O(S).
\end{equation*}
We have
\begin{equation*}
\sum^m_{k=1} \sigma_k n_k^{-(\alpha-\delta)} \le S^{-(\alpha-\delta)}
\Bigg( \sum^m_{k=1} \sigma_k^{\frac{1}{\alpha+1-\delta}}
L_k^{\frac{(\alpha-\delta)s}{\alpha+1-\delta}} \Bigg)^{\alpha+1-\delta}.
\end{equation*}
Let $p\in (1, \decay_{\bsgamma})$, then $\sigma_1 = O(1)$ and
$\sigma_k = O(L_{k-1}^{1-p})$ for $k\ge 2$. Thus we have altogether
\begin{equation*}
e(Q,B(K))^2 = O \left( S^{-(\alpha-\delta)} \left( 1+ L_m^{(\alpha-\delta)s+1-p} \right)
+ L_m^{1-p} \right).
\end{equation*}

\emph{Case 1}: $\decay_{\bsgamma} \ge 1 + \alpha s$. Then we may choose $p$ such that
$(\alpha-\delta)s < p-1$. Choose $m$ such that $S = \Theta(L^{\frac{p-1}{\alpha-\delta}}_m)$, then
\begin{equation*}
e(Q,B(K))^2 = O \left( S^{-(\alpha-\delta)}\right)
= O \left( N^{-(\alpha-\delta)}\right).
\end{equation*}

\emph{Case 2}: $\decay_{\bsgamma} < 1 + \alpha s$. Then, for $\delta$ small enough, we get
$(\alpha-\delta)s > p-1$. Choose $m$ such that $S = \Theta(L^{s}_m)$, then
\begin{equation*}
e(Q,B(K))^2 = O \left( S^{\frac{1-p}{s}}\right)
= O \left( N^{-\frac{p-1}{s}}\right).
\end{equation*}
\end{proof}

\subsubsection{Fixed Subspace Sampling} \label{subsecfixedsubsamp}

In this subsection, we discuss fixed subspace sampling.
For some fixed $L\in\N$ let
\begin{equation*}
v := \bigcup_{j \in [L]} u_j.
\end{equation*}
We focus on algorithms of the form
\begin{equation} \label{eqfixedsubspacesamplalg}
Q(f) = Q_{v,\bsa}(f) = \sum^n_{i=1} w_i  f( \bst^{(i)}_{v} , \bsa).
\end{equation}
We interpret the ``unilevel algorithm'' $Q(f)$ as a multilevel algorithm with
$m=1$, $L_1=L$, $v_1=v$, and $n_1 = n$. Notice that the
upper bound in the next theorem is essentially sharp, as can be seen from
the corresponding lower bound in Corollary \ref{lowboundFIW}.

\begin{theorem} \label{theofixedsubspacesampl} Assume that for $d$ given by \eqref{eqdforfiniteintersectionweights} and all $n\in\N$ there exists an algorithm $Q_n=Q_{[d],n}$
as in the condition of Proposition \ref{Prop3.11} that satisfies (\ref{summe=1}).
Let $Q=Q_{v_1,\bsa}$ be as in Proposition \ref{Prop3.11}.
Let $N$ be the cost of the algorithm corresponding to the cost function
$\$(\nu) = O(\nu^s)$ for some $s>0$. Then there exists for all $\delta >0$ a constant
$C=C(d, \delta)$ such that
\begin{displaymath}
e(Q, B(K))^2 \leq C N^{ - \frac{(\alpha-\delta)(\decay_{\bsgamma}-1) }{(\alpha - \delta)s+\decay_{\bsgamma}-1}}.
\end{displaymath}
\end{theorem}

\begin{proof}
Since our algorithm $Q=Q_{v_1,\bsa}$ is a multilevel algorithm with
$m=1$, we have just to follow the proof of Theorem \ref{Error-FIW}
and modify it slightly.
Let $p\in (1,\decay_{\bsgamma})$. From (\ref{hecke}) and (\ref{minimum}),
we obtain for the choice $n_1 := \lceil x^*_1 \rceil$ 
\begin{displaymath}
e(Q,B(K))^2 = O(S^{-(\alpha - \delta)} L_1^{(\alpha - \delta)s} + L_1^{1-p}),
\end{displaymath}
where $S= \Theta(n_1 L^s_1)$ and $N=O(S)$.
We set
\begin{displaymath}
S := \Theta\left(L_1^{\frac{(\alpha - \delta)s+p-1}{\alpha - \delta}} \right),
\hspace{2ex}\text{resulting in}\hspace{2ex}
e(Q, B(K))^2  =  O \left( N^{-\frac{(\alpha-\delta)(p-1)}{(\alpha-\delta)s + p-1}} \right) \, .
\end{displaymath}
\end{proof}

\subsection{Product Weights}
\label{mlalgprodweightspolys}

In this subsection, we discuss product weights, dealing with variable and fixed subspace sampling separately. We assume that ${\bf a}=(a,a,\dots)$, where $a \in D$ satisfies \eqref{anker_a}. Furthermore, we choose $v_k = v_k^{(2)} = [L_k]$ in (\ref{vk12}).

\subsubsection{Variable Subspace Sampling}

In the variable subspace sampling cost model we have the following result on the multilevel algorithm and product weights.

\begin{theorem} \label{theoprodweightsml} Let $\$(\nu) = O(\nu^s)$ for some $s>0$.
Assume that there exist for every $k \in \mathbb{N}$ algorithms of the form \eqref{algobaustein} that satisfy (\ref{summe=1})
and condition (*) of Lemma \ref{ANOVA} for $\mathcal{V}=\{u\subseteq v_k\}$.
Let $\alpha \ge 1$, and let
$\tau := \min\{\alpha, \decay_{\bsgamma}\} - \delta$ for some $\delta >0$.
Assume further that (\ref{algo-annahme}) holds for all $u \subseteq v_k$
and that for all
$j \in V_k$
\begin{equation} \label{eqdecayweight1}
C_{u_j,k,\tau} \gamma_{u_j} = O( L^{1-p}_{k-1} ) \, ,
\end{equation}
where $p:=\decay_{\bsgamma}-\delta s$, and for all $\emptyset \neq u \subseteq v_{k-1}$
\begin{equation} \label{eqdecayweight2}
C_{u,k,\tau} \gamma_u = O( 1 ) \, .
\end{equation}
Consider the multilevel algorithm defined in \eqref{multilevel-algo}, and let $N$ be the cost of the algorithm corresponding to the cost function
$\$(\nu)$. Then we obtain for $s \ge \frac{\alpha-1}{\alpha}$,
\begin{eqnarray*}
e(Q, B(K))^2 &=& O( N^{- \alpha + \delta}), \textrm{ if }
\decay_{\bsgamma} \geq 1+ \alpha s \, ,
\\ e(Q, B(K))^2 &=& O( N^{- \frac{\decay_{\bsgamma} -1}{s} + \delta}), \textrm{ if } 1+ \alpha s > \decay_{\bsgamma} > 1 \, ,
\end{eqnarray*}
and for $\frac{\alpha-1}{\alpha} > s > 0$,
\begin{eqnarray*}
e(Q, B(K))^2 &=& O( N^{- \alpha + \delta}), \textrm{ if }
\decay_{\bsgamma} \geq \alpha  \, ,
\\ e(Q, B(K))^2 &=& O( N^{- \decay_{\bsgamma} + \delta}), \textrm{ if } \alpha  > \decay_{\bsgamma} > \frac{1}{1-s} \, ,
\\ e(Q, B(K))^2 &=& O( N^{- \frac{\decay_{\bsgamma} -1}{s} + \delta}), \textrm{ if } \frac{1}{1-s} \ge \decay_{\bsgamma} > 1 \, .
\end{eqnarray*}
\end{theorem}

In section \ref{secexample} we will see that condition (\ref{eqdecayweight1}) and
(\ref{eqdecayweight2}) are quite natural and are, in particular, satisfied by scrambled
polynomial lattice rules constructed via a component-by-component approach.

Assume that we have $e_N(B(K_{\{1\}}))^2 = \Omega(N^{-\alpha})$.
 Then we see that for cost functions $\$(\nu) = O(\nu^s)$, where
$s \ge \frac{\alpha-1}{\alpha}$, the upper bounds in Theorem \ref{theoprodweightsml}
are essentially sharp, as confirmed in Corollary \ref{lowboundFIW}. Furthermore, for $\frac{\alpha-1}{\alpha} > s > 0$, the upper bounds are essentially sharp in the regimes $\decay_{\bsgamma} \geq \alpha$ and $\frac{1}{1-s} \ge \decay_{\bsgamma} >1$.
The case $s \ge \frac{\alpha-1}{\alpha}$ is more interesting and relevant in applications
than the case $\frac{\alpha-1}{\alpha} > s > 0$, see, e.g., \cite{Gil08a, NHMR11, PW11}.

\begin{proof}
Let
$L_k = L\lceil A^k \rceil$ for fixed $L\in\N$ and $A>1$.
We use the analysis from Subsection \ref{subsecgenweights} and get \eqref{Fehlerabschaetzung}.
We have
\begin{eqnarray*}
r^2_{k,u_j} &=& \sum_{w \subset \mathbb{N} \setminus v_k} \gamma_{u_j \cup w} k_w(\bsa, \bsa) = \gamma_{u_j} \sum_{w \subset \mathbb{N} \setminus v_k} \gamma_w k_w(\bsa,\bsa) = O ( \gamma_{u_j} )
\end{eqnarray*}
and
\begin{eqnarray*}
\tilde{r}^2_{k,u} &=& \tilde{r}^2_{v_k,v_{k-1},u,\bsa} = \sum_{w' \subset \mathbb{N} \setminus v_{k-1}; w' \cap v_k \neq \emptyset } \gamma_{u \cup w'} k_{w'}(\bsa,\bsa)
\\ &\leq& \gamma_{u} \sum_{\emptyset \neq w' \subset \mathbb{N} \setminus v_{k-1}} \gamma_{w'} k(a,a)^{\vert w' \vert} \, .
\end{eqnarray*}
Now we have
\begin{equation*}
\begin{split}
&\sum_{\emptyset \neq w' \subset \mathbb{N} \setminus v_{k-1}} \gamma_{w'}
k(a,a)^{\vert w' \vert}
= \prod_{j =  L_{k-1}+1}^\infty (1+\gamma_j k(a,a)) -1\\
\le &k(a,a)\! \bigg(\sum_{j =  L_{k-1}+1}^\infty \gamma_j \bigg)\!
\Bigg( 1 + \frac{k(a,a)}{2} \bigg( \sum_{j =  L_{k-1}+1}^\infty \gamma_j \bigg)
\exp \bigg( k(a,a) \sum_{j =  L_{k-1}+1}^\infty \gamma_j \bigg) \Bigg)\\
= &k(a,a) O \left( L_{k-1}^{1-p} \right) \big( 1+o(1) \big) =  O \left( L_{k-1}^{1-p} \right).
\end{split}
\end{equation*}
From this we  obtain
\begin{equation*}
\tilde{r}^2_{k,u} = O( \gamma_u L^{1-p}_{k-1} )
\hspace{3ex}\text{and}\hspace{3ex} {\rm b}^2_m = O(L^{1-p}_m).
\end{equation*}
Thus we have from Equations \eqref{eqdecayweight1} and \eqref{eqdecayweight2}
\begin{equation} \label{eqprodweightsmlerror}
e(Q, B(K))^2 = O \left( \sum^m_{k=1}
L^{1 - p}_{k-1} (n_k-1)^{-\tau} + L^{1-p}_m \right) \, .
\end{equation}
We use the notation $\sigma_k := L^{1-p}_{k-1}$ and
$M:=\sum^m_{k=1} L^s_k$. Let $S \geq 2M$ be given.
Arguing as in the proof of Theorem \ref{Error-FIW}, we choose $n_k = \lceil x^*_k \rceil + 1$, where
\begin{equation} \label{eqprodweightsx}
x^*_k =  C \left( L^{\frac{1-p-s}{\tau+1}} \right) 
\hspace{2ex}\text{and}\hspace{2ex}
C = S \left( \sum^m_{k=1}  L^{\frac{1 - p + s \tau}{\tau +1}}_k \right)^{-1} \, .
\end{equation}
This leads to
\begin{displaymath}
N = O \left( \sum^m_{k=1} n_k L^s_k \right) = O(S) \, .
\end{displaymath}
We now have
\begin{eqnarray*}
\sum^m_{k=1} L^{1-p}_{k-1} (n_k-1)^{-\tau} &\leq&
S^{-\tau} \left( \sum^m_{k=1} L^{\frac{1 - p + s \tau}{\tau +1}}_k \right)^{\tau +1}
\\ \\ &= & O \left( S^{-\tau} \left( 1 + L^{1 - p + s \tau}_m \right)\right) \, .
\end{eqnarray*}
We set $S= \Theta(L^s_m)$ and obtain
\begin{eqnarray*}
e \left( Q, B(K) \right)^2 &=& O \left( S^{-\tau} + S^{- \tau} L^{1 - p + s \tau}_m + L^{1-p}_m \right)
\\ &=& O \left( S^{- \tau} + S^{\frac{1-p}{s}} \right)
\\ &=& O \left( S^{- \min \left\{ \tau, \frac{p-1}{s} \right\} } \right)
\\ &=& O \left( S^{- \min \left\{ \min \left\{ \alpha , \decay_{\bsgamma} \right\} , \frac{\decay_{\bsgamma} - 1}{s} \right\} + \delta } \right) \, .
\end{eqnarray*}
We consider two cases, $s \ge \frac{\alpha-1}{\alpha}$ and
$\frac{\alpha-1}{\alpha} > s > 0$.

\emph{Case 1}: If $s \ge \frac{\alpha-1}{\alpha}$, we consider two subcases.
If $\decay_{\bsgamma} \geq 1 + \alpha s$, then we have
$\frac{\decay_{\bsgamma}-1}{s} \ge \alpha$ and $\decay_{\bsgamma} \ge \alpha$. Hence
\begin{displaymath}
e ( Q, B(K))^2
= O \left( S^{- \alpha + \delta} \right) \, .
\end{displaymath}
If $1 + \alpha s > \decay_{\bsgamma} > 1$, then
$\frac{\decay_{\bsgamma}-1}{s} \in (0, \alpha)$. Moreover, $s\ge \frac{\alpha -1}{\alpha}$
implies $\decay_{\bsgamma} \ge \frac{\decay_{\bsgamma}-1}{s}$. Thus
\begin{displaymath}
e(Q, B(K))^2 = O \left( S^{ - \frac{\decay_{\bsgamma}-1}{s} + \delta} \right) \, .
\end{displaymath}

\emph{Case 2}: If $ \frac{\alpha-1}{\alpha} > s > 0$, we consider three subcases.
If $\decay_{\bsgamma} \geq \alpha$, then $\frac{\decay_{\bsgamma}-1}{s} > \alpha$.
Hence
\begin{displaymath}
e ( Q, B(K))^2 = O \left( S^{- \alpha + \delta} \right) \, .
\end{displaymath}
If $\alpha > \decay_{\bsgamma} > \frac{1}{1-s}$, then
$\frac{\decay_{\bsgamma}-1}{s} > \decay_{\bsgamma}$, and
\begin{displaymath}
e( Q, B(K))^2 = O \left( S^{- \decay_{\bsgamma} +\delta} \right).
\end{displaymath}
Finally, if $ \frac{1}{1-s} \ge \decay_{\bsgamma} >1$, then
$\decay_{\bsgamma} \ge \frac{\decay_{\bsgamma}-1}{s}$, and
\begin{displaymath}
e( Q, B(K))^2 = O \left( S^{- \frac{ \decay_{\bsgamma} -1}{s} +\delta} \right) \, .
\end{displaymath}
%
%
%
\end{proof}

\subsubsection{Fixed Subspace Sampling}

For fixed subspace sampling, our analysis recovers Theorem 1 from \cite{HMGNR10}.

\section{Example: The Unanchored Sobolev Space and Scrambled Polynomial Lattice Rules} \label{secexample}

In this section, we apply the results from Section \ref{secmultilevelalg} to a particular function space, the unanchored Sobolev space, and employ a particular class of quadrature rules, namely scrambled polynomial lattice rules.

\subsection{The Unanchored Sobolev Space} \label{subsecunachoredSobspace}

We recall the \emph{unanchored Sobolev space}, which is also discussed for example in \cite{HMGNR10,NW08,YH05}. Let $k:[0,1]^2\to \rr$ be the \emph{unanchored kernel} given by
\begin{equation*}
k(x,y) = 1/3 + (x^2 + y^2)/2 - \max(x,y).
\end{equation*}
Regarding the anchor, we fix $a=\frac{1}{2}$ and set $\bsa=(a,a,\dots)$ which minimizes
\begin{displaymath}
\sum_{u \in \mathcal{U}} \widehat{\gamma}_u = \sum_{u \in \mathcal{U}} \gamma_u k_{u} (\bsa, \bsa) \, ,
\end{displaymath}
see e.g. \cite{HMGNR10}. For $u \neq \emptyset$ the space $H_{u}$ consists of all absolutely continuous functions $f$ such that the weak derivative $f^{(u)} = \frac{\partial^{\vert u \vert} }{\prod_{j \in u} \partial x_j} f$ satisfies $f^{(u)} \in L_2 ([0,1]^u)$ and $\int_0^1 f(\bsy)\,{\rm d} y_j = 0$ for all $j\in u$. We have
\begin{displaymath}
\| f \|^2_{k_u} = \int_{[0,1]^u} ( f^{(u)} (\bsy) )^2 \,{\rm d}\bsy \, .
\end{displaymath}

\begin{remark} \label{remmininterrsUnanchoredSobspace}
We recall from \cite{N88}, Section 2.2.9, Proposition 1, that the $N$th minimal integration error in the Sobolev space $W^1_2 ( [0,1] )$ is of order $\Omega (N^{-3/2})$. For $\gamma_{\left\{ 1 \right\}} >0 $ the space $W^1_2 ([0,1])$ is obviously continuously embedded in $H(K_{ \left\{ 1 \right\}})$, thus implying $e_N (B(K_{\left\{ 1 \right\}}) )^2 = \Omega (N^{-3})$.
\end{remark}

\subsection{Scrambled Polynomial Lattice Rules} \label{subsecscrambbledpolylattrules}

In this subsection we recall a result on scrambled polynomial lattice rules that will be used in Subsection \ref{subsecresultsmultilevelalgs}. Polynomial lattice rules were introduced in \cite{N92a}, see also \cite{DKPS05,DP09,N92}. For background on the scrambling algorithm, we refer the reader to \cite{O95,O97}, for background on scrambled polynomial lattice rules and finite-dimensional integration results, we refer the reader to \cite{BD10}. The proof of the following result is given in the Appendix.
\begin{theorem} \label{theoCBC}
Let $(\gamma_u)_{u\in\U}$ be general weights. Assume that $v_k$, for $k \in \mathbb{N}$,
is as in Section \ref{secmultilevelalg}. Then for all $k$, $M_k \in \mathbb{N}$, there exists a scrambled polynomial lattice rule $\left\{ \bsx_i \right\}^{n_k}_{i=1} \in [0,1)^{v_k}$, where $n_k = b^{M_k}$ and $b$ a prime, such that the algorithm
\begin{equation} \label{eqalgCBCSec52}
Q_{v_k,n_k}(f) = \frac{1}{n_k} \sum^{n_k}_{i=1} f(\bsx_i) \, ,
\end{equation}
satisfies condition (*) of Lemma \ref{ANOVA} for $\mathcal{V} = \{u\subseteq v_k\}$,
and we have for all $ 1\leq \tau <3$ and all $u \subseteq v_k$,
\begin{displaymath}
\Var(Q_{v_k,n_k} (f_u)) \leq  C_{u,k,\tau} (n_k-1)^{- \tau} \|f_u \|^2_{k_u}
\hspace{3ex}\text{for all $f_u\in H_u$,}
\end{displaymath}
where
\begin{displaymath}
C_{u,k,\tau} = \left( \sum_{z_u \in w \subseteq [z_u]} \gamma^{\frac{1}{\tau}}_w C^{\vert w \vert}_{b, \frac{1}{\tau}} \right)^{\tau}  \gamma^{-1}_u \, ,
\end{displaymath}
here $C_{b, \lambda}$ is given by
\begin{equation} \label{eqCconstCBC}
C_{b, \lambda} := \max \left( \frac{b-1}{3^{\lambda} } \frac{b^{3 \lambda -1}}{b^{3 \lambda -1}-1} , \frac{b^{2 \lambda}}{(b+1)^{\lambda} 3^{\lambda}} \right) \, ,
\end{equation}
and $z_u = \max u$, $u \subset \mathbb{N}$.
\end{theorem}

We remark that the scrambled polynomial lattice rules referred to in Theorem \ref{theoCBC} can be constructed using a modification of the component-by-component (CBC) algorithm from \cite{BD10}, see the Appendix.

\subsection{Results for multilevel algorithms} \label{subsecresultsmultilevelalgs}

In this subsection, we present results for the multilevel algorithms from Section \ref{secmultilevelalg} for the space of integrands $H(K)$ based on the unanchored
kernel $k$ discussed in Subsection \ref{subsecunachoredSobspace}. We rely on the scrambled polynomial lattice rules from Theorem \ref{theoCBC}. We remark that scrambled polynomial lattice rules consist of $n$ points, where $n$ is the power of a prime, see Theorem \ref{theoCBC}, and cannot be constructed for all $n \in \mathbb{N}$, as stated in the propositions and theorems of Section \ref{secmultilevelalg}. However, when required to construct a quadrature rule consisting of $n$ points, where $n \in \mathbb{N}$, we simply construct a scrambled polynomial lattice rule consisting of $b^M$ points, where $b^M \leq n < b^{M+1}$, and we set the quadrature weights corresponding to the superfluous points equal to zero.

\subsubsection{Finite-Intersection Weights}

We now present results for finite-inter\-sec\-tion weights distinguishing between variable and fixed subspace sampling.

\subsubsection*{Variable Subspace Sampling}

For finite-intersection weights and variable subspace sampling, we have the following result, which is essentially optimal in the case where the monotonicity condition (\ref{monoton}) holds and $\gamma_{\{1\}} >0$, see Corollary \ref{lowboundFIW} and Remark \ref{remmininterrsUnanchoredSobspace}. Let again $d$ be as in (\ref{eqdforfiniteintersectionweights})

\begin{corollary} \label{corrfiniteintersectionweightsvarsubspace}
Let $\$(\nu) = O(\nu^s)$ for $s>0$. Let $Q_n = Q_{[d], n}$ be the algorithm $Q_{v_1,n_1}$
from Theorem \ref{theoCBC} with $n_1=n$ and $v_1=[d]$.
Consider the multilevel algorithms $Q$ defined in (\ref{multilevel-algo}),
where the algorithms $Q_{v_k, \bsa}$ are as in Proposition \ref{Prop3.11}
Put $N:= \cost_{\var}(Q, B(K))$.
Then there exists for all $\delta >0$ a constant $C=C(d,\delta)$ such that
\begin{equation*}
e(Q,B(K))^2 \le C N^{-3+\delta}
\hspace{2ex}\text{if $\decay_{\bsgamma} \ge 1+3 s$,}
\end{equation*}
and
\begin{equation*}
e(Q, B(K))^2 \le C N^{-\frac{\decay_{\bsgamma}-1}{s} +\delta}
\hspace{2ex}\text{if $ 1+ 3 s > \decay_{\bsgamma} > 1 $.}
\end{equation*}
\end{corollary}

\begin{proof} We need to verify that the conditions of Theorem \ref{Error-FIW} are satisfied. Scrambled polynomial lattice rules $Q_{[d],n}$ are of the form \eqref{alg-form}, satisfy condition (*) of Lemma \ref{ANOVA} for $\mathcal{V}:=\{u\subseteq [d]\}$, and \eqref{eqconditionquadrulefinintersecweights} for arbitrarily small $\varepsilon>0$,
$\alpha = 3-\varepsilon=:\tau$, $\beta =0$ and
\begin{displaymath}
C_d = \left( \sum_{\emptyset \neq w \subseteq [d]} C^{\vert w \vert}_{b, \frac{1}{\tau}} \right)^{\tau} \, ,
\end{displaymath}
where $C_{b, \lambda}$ is given by \eqref{eqCconstCBC} and we simply set $\gamma_u =1$ for all $u \subseteq [d]$ in Theorem \ref{theoCBC}.
\end{proof}

\subsubsection*{Fixed Subspace Sampling}

For fixed subspace sampling, we obtain from Theorem \ref{theofixedsubspacesampl}
the following result, which is essentially optimal according to Corollary \ref{lowboundFIW} and Remark \ref{remmininterrsUnanchoredSobspace}.

\begin{corollary} Let $\$(\nu) = O(\nu^s)$ for some $s>0$.
Let $Q_n=Q_{[d],n}$ be the algorithm $Q_{v_1, n_1}$
from Theorem \ref{theoCBC}, where $n_1=n$ and $v_1=[d]$. Let $Q=Q_{v_1, \bsa}$ be as in Proposition \ref{Prop3.11}.
Put $N:= \cost_{\var}(Q, B(K))$.
Then there exists for all $\delta >0$ a constant $C=C(d,\delta)$ such that
\begin{equation*}
e(Q,B(K))^2 \le C N^{-\frac{(3 - \delta)(\decay_{\bsgamma} - 1)}{(3 - \delta)s + \decay_{\bsgamma} -1}} \, .
\end{equation*}
\end{corollary}

\subsubsection{Product Weights}

For product weights, we have the following results, where we again distinguish between variable and fixed subspace sampling.

\subsubsection*{Variable Subspace Sampling}

For variable subspace sampling we obtain the following result, which is essentially optimal for cost functions $ \$ (\nu) = O(\nu^s)$, that satisfy $s \ge \frac{2}{3}$. For $\frac{2}{3} > s >0$, the results are optimal for the regimes $\decay_{\bsgamma} \geq 3$ and $\frac{1}{1-s} \ge \decay_{\bsgamma} >1$, see Corollary \ref{lowboundFIW} and Remark \ref{remmininterrsUnanchoredSobspace}.

\begin{corollary} \label{theomultilevelprodweights} Let $\left( \gamma_{u} \right)_{u \in \mathcal{U}}$ be product weights and consider the algorithm
\begin{displaymath}
Q(f) = \sum^{m}_{k=1} Q_{v_k,\bsa}(f - \Psi_{k-1} f) \, ,
\end{displaymath}
where the $Q_{v_k,\bsa}$ are related to the scrambled polynomial lattice rules
$Q_{v_k,n_k}$ from Theorem \ref{theoCBC} via $Q_{v_k,\bsa} = Q_{v_k,n_k} \circ \Psi_k$.
Let $N$ be the cost of the algorithm corresponding to the cost function $ \$ (\nu) = O(\nu^s)$ for some $s > 0$. Then, for arbitrarily small $\delta>0$, we obtain for $s \ge \frac{2}{3}$,
\begin{eqnarray*}
e(Q, B(K))^2 &=& O( N^{- 3 + \delta}), \textrm{ if }
\decay_{\bsgamma} \geq 1+ 3 s \, ,
\\ e(Q, B(K))^2 &=& O( N^{- \frac{\decay_{\bsgamma} -1}{s} + \delta}), \textrm{ if } 1+ 3 s > \decay_{\bsgamma} > 1 \, ,
\end{eqnarray*}
and for $\frac{2}{3} > s > 0$,
\begin{eqnarray*}
e(Q, B(K))^2 &=& O( N^{- 3 + \delta}), \textrm{ if }
\decay_{\bsgamma} \geq 3  \, ,
\\ e(Q, B(K))^2 &=& O( N^{- \decay_{\bsgamma} + \delta}), \textrm{ if } 3  > \decay_{\bsgamma} > \frac{1}{1-s} \, ,
\\ e(Q, B(K))^2 &=& O( N^{- \frac{\decay_{\bsgamma} -1}{s} + \delta}), \textrm{ if } \frac{1}{1-s} \ge \decay_{\bsgamma} > 1 \, .
\end{eqnarray*}
\end{corollary}

Before proving Corollary \ref{theomultilevelprodweights}, we compare it to the results obtained in \cite{B10} and \cite{HMGNR10}, where the case $s=1$ was treated. So let $s=1$. In \cite{HMGNR10}, the rate of convergence $3-\delta$, $\delta$ arbitrarily small,
was achieved for $\decay_{\bsgamma} \geq 11$, see Corollary 4, and in \cite{B10} for $\decay_{\bsgamma} \geq 10$. Using our analysis, we achieve this rate for $\decay_{\bsgamma} \ge 4$, a result which is optimal, see Corollary \ref{lowboundFIW}
and Remark \ref{remmininterrsUnanchoredSobspace}, and thus cannot be improved further.
In the remaining regime $4> \decay_{\bsgamma}>1$, the result of Corollary \ref{theomultilevelprodweights} is again essentially optimal
and improves clearly on the results from \cite{B10} and \cite{HMGNR10}.

\begin{proof} We need to verify the conditions of Theorem \ref{theoprodweightsml}. Since we base the algorithms $Q_{v_k,\bsa}$ on the scrambled polynomial lattice rules from Theorem \ref{theoCBC}, we obtain \eqref{algo-annahme} for all $1 \leq \tau < 3$. We now confirm that the constants $C_{u,k,\tau}$ satisfy Equations \eqref{eqdecayweight1} and \eqref{eqdecayweight2} respectively. For $u\in v_k$, we have
\begin{eqnarray*}
 \gamma_{u} \left( \sum_{ z_{u} \in w \subseteq [z_{u}] } \gamma^{\frac{1}{\tau}}_w  C^{\vert w \vert}_{b, \frac{1}{\tau}} \right)^{\tau} \gamma^{-1}_{u}  &=&  \left( \sum_{ z_{u} \in w \subseteq [z_{u}]  } \gamma^{\frac{1}{\tau} }_w C^{\vert w \vert}_{b, \frac{1}{\tau}} \right)^{\tau}   \, .
\end{eqnarray*}
For $\tau \in (1, \decay_{\bsgamma} )$, since we deal with product weights, we obtain
for $j\in V_k$
\begin{displaymath}
\sum_{z_{u_j} \in w \subseteq [z_{u_j}]} \gamma^{\frac{1}{\tau}}_w C^{\vert w \vert}_{b, \frac{1}{\tau}} = \gamma^{\frac{1}{\tau}}_{z_{u_j}} C_{b, \frac{1}{\tau}} \sum_{w \subseteq [z_{u_j}-1]} \gamma^{\frac{1}{\tau}}_w C^{\vert w \vert}_{b, \frac{1}{\tau}} = O( \gamma^{\frac{1}{\tau}}_{z_{u_j}} ) \, ,
\end{displaymath}
and hence
$C_{u_j,k,\tau} \gamma_{u_j} = O \left( L^{-p}_{k-1} \right)$
for each $p< \decay_{\bsgamma}$. For $\emptyset \neq u \subseteq v_{k-1}$,
\begin{displaymath}
\left( \sum_{z_u \in w \subseteq [z_u]} \gamma^{\frac{1}{\tau}}_w C^{\vert w \vert}_{b, \frac{1}{\tau}} \right)^{\tau} \leq \left( \sum_{ w \in \mathcal{U} } \gamma^{\frac{1}{\tau}}_w C^{\vert w \vert}_{b, \frac{1}{\tau}} \right)^{\tau} < \infty \, ,
\end{displaymath}
so
$C_{u,k,\tau} \gamma_u = O \left( 1 \right)$
as required.
\end{proof}

\subsubsection*{Fixed Subspace Sampling}

We can combine Theorem 1 from \cite{HMGNR10} with scrambled polynomial lattice rules to recover Corollary 3.1 from \cite{B10}, which used Theorem 1 from \cite{HMGNR10} to improve on Corollary 1 from \cite{HMGNR10}.

\section*{Acknowledgments} Michael Gnewuch was partially supported by the German Science Foundation DFG under grant GN91-3/1 and by the Australian Research Council ARC.

\appendix

\section{Scrambled Polynomial Lattice Rules} \label{subsecdefandprelimpolys}

The quadrature rules employed in Section \ref{secexample} are based on scrambled polynomial lattice rules, which we now recall mostly relying on \cite{BD10}.

Polynomial lattice rules were introduced in \cite{N92a}, see also \cite{DKPS05,DP09,N92}. We recall that $b$ is a fixed prime and denote by $\zz_b$ the finite field containing $b$ elements and by $\zz_b((x^{-1}))$ the field of formal Laurent series over $\zz_b$. Elements of $\zz_b((x^{-1}))$ are formal Laurent series,
\begin{displaymath}
 L=\sum^{\infty}_{l=w} t_l x^{-l} \, ,
\end{displaymath}
where $w$ can be an arbitrary integer and all $t_l$ are in $\zz_b$. The field $\zz_b ((x^{-1}))$ contains the field of rational functions over $\zz_b$ as a subfield. Finally, the set of polynomials over $\zz_b$ is denoted by $\zz_b[x]$. For an integer $M$, we denote by $\nu_M$ the map from $\zz_b((x^{-1}))$ to $[0,1)$ defined by
\begin{displaymath}
 \nu_M \left( \sum^{\infty}_{l=w} t_l x^{-l} \right) = \sum^{M}_{l=\max(1,w)} t_l b^{-l} \, .
\end{displaymath}
The following definition of polynomial lattice rules stems from \cite{N92a}, see also \cite{DP09, N92}. Recall that a \emph{quasi-Monte Carlo rule} is a linear
quadrature rule whose quadrature weights are all equal and sum up to one.

\begin{definition} \label{defpolylatt} Let $b$ be prime and $M$ be an integer. For a given dimension $s \geq 1$, choose $p(x) \in \zz_b[x]$ with $deg(p(x))=M$ and $q_1(x), \dots, q_s(x) \in \zz_b[x]$. For $0 \leq h < b^M$ let $h=h_0 + h_1 b + \dots + h_{M-1} b^{M-1}$ be the $b$-adic expansion of $h$. With each such $h$ we associate the polynomial
\begin{displaymath}
 \overline{h}(x) = \sum^{M-1}_{r=0} h_r x^r \in \zz_{b}[x] \, .
\end{displaymath}
Then $S_{p,M}(\bsq)$, where $\bsq=(q_1,\dots,q_s)$, is the point set consisting of the $b^M$ points
\begin{displaymath}
\bsx_h = \left( \nu_M \left(\frac{ \overline{h}(x) q_1(x) }{p(x)} \right) , \dots, \nu_M \left( \frac{ \overline{h}(x) q_s(x) }{p(x)} \right) \right) \in [0,1)^s \, ,
\end{displaymath}
for $0 \leq h < b^M$. A quasi-Monte Carlo rule using the point set $S_{p,M}(\bsq)$ is called a polynomial lattice rule.
\end{definition}

Regarding notation, we write
$\bsh$ for vectors over $\zz$ or $\rr$. Polynomials over $\zz_b$ are denoted by $h(x)$ and vectors of polynomials by $\bsh(x)$.
Finally, we introduce the dual lattice which plays an important role in numerical integration, see \cite{DKPS05,DP09}, which requires us to introduce the following function: for a non-negative integer $k$ with $b$-adic expansion $k=k_0 + k_1 b + \dots $ we write ${\rm tr}_M(k) = k_0 + k_1 b + \dots + k_{M-1} b^{M-1} $ and thus the associated polynomial
\begin{displaymath}
 {\rm tr}_{M}(k)(x) = k_0 + k_1 x + \dots k_{M-1} x^{M-1} \in \zz_b[x]
\end{displaymath}
has degree $< M$. For a vector $\bsk \in \nn^s_0$, ${\rm tr}_{M} (\bsk)$ is defined componentwise.
We fix a polynomial $p(x) \in \zz_b[x]$ with $deg(p(x))=M$.

\begin{definition} \label{defdualpollatt} Let $\bsq(x)=(q_1(x), \dots,q_s(x)) \in \zz_b[x]^s$, then the dual polynomial lattice of $S_{p,M}( \bsq)$ is given by
\begin{eqnarray*}
\mathcal{D}  =  \mathcal{D}_{p}(\bsq)    =  \left\{ \bsk \in \nn^s_0:\; \sum^s_{j=1} {\rm tr}_{m}(k_j)(x)q_j(x) \equiv \bs0 \pmod{ p(x) } \right\}.
\end{eqnarray*}
\end{definition}
Also, we set $\mathcal{D}'=\mathcal{D} \setminus \left\{ \bs0 \right\}$ and use the notation $\mathcal{D}_p(\bsq_{u})$ to denote the dual lattice corresponding to the generating polynomials $q_j$, $j \in u$, and define $\mathcal{D}'(\bsq_u)$ analogously. The following function plays an important role in the analysis of polynomial lattice rules

\begin{displaymath}
r(l) = \left\{ \begin{array}{cc} 1 & \textrm{ if } l=0 \\  \frac{1}{3 b^{3 a}} & \textrm{ if } l >0 \end{array} \right. \, ,
\end{displaymath}
where $l=l_0 + l_1 b + \dots + l_a b^a$, $l_a \neq 0$, and for $\bsl_{u} \in \nn^{ \vert u \vert}_0$ we set $r_{\bsgamma} (\bsl_{u}) = \gamma_{u_h} \prod_{j \in u_h} r(l_j)$, where $u_h = \left\{ j \in u : l_j > 0  \right\}$.

We now recall \emph{Owen's scrambling algorithm} introduced in \cite{O95, O97}. The scrambling algorithm is best illustrated for a generic point $\bsx \in [0,1)^s$, where $\bsx=(x_1,\dots,x_s)$ and
\begin{displaymath}
 x_j = \frac{\xi_{j,1}}{b} + \frac{\xi_{j,2}}{b^2} + \dots .
\end{displaymath}
Then the scrambled point shall be denoted by $\bsy \in [0,1)^s$, where $\bsy=(y_1,\dots,y_s)$,
\begin{displaymath}
 y_j=\frac{\eta_{j,1}}{b} + \frac{\eta_{j,2}}{b^2} + \dots .
\end{displaymath}
The permutation applied to $\xi_{j,l}$, $j=1,\dots,s$, depends on $\xi_{j,k}$, for $1 \leq k <l$. In particular, $\eta_{j,1}=\pi_j(\xi_{j,1})$, $\eta_{j,2}=\pi_{j, \xi_{j,1}}(\xi_{j,2})$, $\eta_{j,3} =\pi_{j,\xi_{j,1}, \xi_{j,2}}(\xi_{j,3})$ and in general
\begin{displaymath}
 \eta_{j,k} = \pi_{j,\xi_{j,1} , \dots, \xi_{j,k-1} } (\xi_{j,k}) \, , \, k \geq 2 \, ,
\end{displaymath}
where $\pi_j$ and $\pi_{j, \xi_{j,1} , \dots , \xi_{j,k-1}}$, $k \geq 2$, are random permutations of $\left\{ 0, \dots, b-1 \right\}$. We assume that permutations with different indices are mutually independent. Using $\mathcal{P}$ to denote a point set in $[0,1)^s$ and $\mathcal{P}_{\pi}$ to denote the point set resulting from the application of Owen's scrambling algorithm to the points in $\mathcal{P}$, it is known from \cite{O95}, see Proposition 2, that each point in $\mathcal{P}_{\pi}$ is uniformly distributed in $[0,1)^s$. Using Owen's scrambling algorithm to randomize polynomial lattice rules, we are able to obtain the following estimate on the variance of a quadrature rule based on a scrambled polynomial lattice rule.

\begin{theorem} \label{corrvarestpolysprodweights} Let $b$ be prime, $M$ an integer and set $n=b^M$. Assume $s \in \mathbb{N}$ and that $\left( \gamma_{u} \right)_{u \in \mathcal{U}}$ are general weights. We set
\begin{equation} \label{eqscrpolylattrule}
Q_{[s],n} (f) = \frac{1}{n} \sum^n_{i=1} f(\bsx_i) \, ,
\end{equation}
where $\left\{ \bsx_i \right\}^{n}_{i=1} \in [0,1)^s$ is based on a scrambled polynomial lattice rule, and obtain for all $u \subseteq [s]$
\begin{equation} \label{eqVarboundscrpolylattrule}
\Var(Q_{[s],n}(f_{u})) \leq \sum_{\stackrel{\bsl_{u} \in \nn^{\vert u \vert}}{\bsl_{u} \in \mathcal{D}(\bsq_{u})}} r_{\bsgamma} (\bsl_{u}) \gamma^{-1}_u \| f_{u} \|^2_{k_{u}} \, .
\end{equation}
\end{theorem}


Theorem \ref{corrvarestpolysprodweights} can be verified by recalling that polynomial lattice rules are digital nets, see e.g. \cite{DP09, N92}, and using the proof approach of \cite[Corollary 13.7]{DP09}. The coefficients $\sigma^2_{(\bsl_u , \bszero)} (f)$ appearing in that proof can be bounded in terms of the norm $\| \cdot \|_{k_u}$ with the help of the analysis in the proof of \cite[Lemma 6]{YH05}.

For $u\subseteq v_k$ we can use the bound
\begin{eqnarray} \label{eqqualitycrit1}
\sum_{ \stackrel{\bsl_u \in \mathbb{N}^u}{\bsl_u \in \mathcal{D}(\bsq_u)}} r_{\bsgamma}(\bsl_u) &\leq& \sum_{z_u \in w \subseteq [z_u]} \sum_{ \stackrel{\bsl_w \in \mathbb{N}^w}{\bsl_w \in \mathcal{D}(\bsq_w)} } r_{\bsgamma}(\bsl_w)=: B(\bsq, [z_u], \tilde{\bsgamma}^{(z_u)} ) \, , 
\end{eqnarray}
where
\begin{displaymath}
\tilde{\gamma}^{(z_u)}_w = 0 \textrm{ for } w \subseteq [z_u -1] \textrm{ and } \tilde{\gamma}^{(z_u)}_w = \gamma_w \textrm{ for } z_u \in w \subseteq [z_u] \, ,
\end{displaymath}
where $z_u = \max u$, $u \subset \mathbb{N}$.
%

\section{Constructing Polynomial Lattice Rules for the Unanchored Sobolev Space} \label{subsecconstrpolylatt}

The aim of this section is twofold: Firstly, we would like to discuss how to implement the multilevel algorithm from Section \ref{secexample} in practice, and secondly we would like to establish Theorem \ref{theoCBC}. The construction of the scrambled polynomial lattice rules underlying the algorithm from Section \ref{secmultilevelalg} is based on the component-by-component (CBC) construction from \cite{BD10}. In \cite{BD10}, the construction was presented in the context of a Walsh function space and product weights, whereas we are going to present the results for the unanchored Sobolev space from Subsection \ref{subsecunachoredSobspace} and general weights.

We will illustrate how to construct the scrambled polynomial lattice rule underlying the algorithm $Q_{v_k, \bsa}$, see Equation \eqref{algobaustein}. To do so, we proceed as follows: We note that the sets $ ( v_k)_{k \in \mathbb{N}}$ from Section \ref{secmultilevelalg} satisfy $v_1 \subseteq v_2 \subseteq v_3 \subseteq \dots$, hence we firstly construct a scrambled polynomial lattice rule corresponding to the set $v_1$, i.e. construct a point set in $[0,1)^{v_1}$. Subsequently, we extend this point set to a scrambled polynomial lattice rule corresponding to the set $v_2$, i.e. construct points in $[0,1)^{v_2}$. Next we extend this point set to a scrambled polynomial lattice rule corresponding to the set $v_3$, i.e. construct points in $[0,1)^{v_3}$, etc..
Hence we need to present two algorithms; the first algorithm, \emph{CBC 1}, shows how to construct scrambled polynomial lattice rules corresponding to $v_1$ in $[0,1)^{v_1}$, the next algorithm, \emph{CBC 2}, shows how to extend a scrambled polynomial lattice rule corresponding to $v_{k-1}$ in $[0,1)^{v_{k-1}}$ to a scrambled polynomial lattice rule corresponding to $v_k$ in $[0,1)^{v_k}$, for $k=2,3, \dots$. Clearly it suffices to show how to construct the polynomial lattice rule corresponding to $v_1$ and then how to extend it to a polynomial lattice rule corresponding to $v_2$. Without loss of generality, we assume that $v_1 = [s_1]$ and $v_2 = [s_2]$, where $s_1, s_2 \in \mathbb{N}$, and $s_1 < s_2$. We now discuss how to construct a scrambled polynomial lattice rule in $[0,1)^{s_1}$ using the CBC construction. Intuitively, the CBC construction chooses the polynomials $q_1, \dots, q_{s_1}$ in a greedy fashion: The first polynomial $q_1$ is chosen so that a given quality criterion is minimized. The resulting polynomial is then fixed, say $q^*_1$, and consequently the second polynomial $q_2$ is chosen so that the quality criterion is minimized. The resulting polynomial, say $q^*_2$, is now fixed and we continue this procedure. We note that the quality criterion plays a crucial role for the CBC construction, and we use $B( \bsq, [z_u], \tilde{\bsgamma}^{(z_u)} )$ from \eqref{eqqualitycrit1} as quality criterion to construct the scrambled polynomial lattice rule in $[0,1)^{z_u}$. This criterion is closely related to the bound on the variance from Equation \eqref{eqVarboundscrpolylattrule}, and we have for all $u \subseteq [s_1]$
\begin{equation} \label{eqVarboundscrpolylattrule2}
\Var(Q_{[s_1],n} (f_u)) \leq B(\bsq, [z_u], \tilde{\bsgamma}^{(z_u)}) \gamma^{-1}_u \| f_u \|^2_{k_u}
\hspace{3ex}\text{for all $f_u \in H_u$.}
\end{equation}
To be useful for the CBC 1 algorithm, we need the quality criterion $B(\bsq, [z_u], \tilde{\bsgamma}^{(z_u)})$ to be computable. The following theorem provides an explicit formula, see \cite[Lemma~1]{BD10}.

\begin{theorem} \label{theocompB} Let $b$ be prime, $M$ an integer and set $n=b^M$. Then the following equality holds for $\bsq = (q_1, \dots, q_{[z_u]})$ and the point set $S_{p,M}( \bsq) = \{\bsx_i\}_{i=1}^{b^M}$ for $u \subseteq [s_1]$
\begin{equation} \label{eqdefB}
B(\bsq , [z_u], \tilde{\bsgamma}^{(z_u)} ) = \frac{1}{n} \sum^n_{i=1} \sum_{z_u \in w \subseteq [z_u]} \gamma_w \prod_{j \in w} \left( \frac{\phi ( x_{i,j} )}{3} \right) \, ,
\end{equation}
where
\begin{displaymath}
\phi(x) = \left\{ \begin{array}{ll} \frac{b^2}{b+1} & \textrm{ for } x=0 \\ \frac{b^2(1-b^{-2 (a_0 -1)})}{(b+1)} - b^{-2 a_0} b^2 & \textrm{ for } \xi_i=0 \, , \, i=1, \dots , a_0-1 \, , \, \xi_{a_0} \neq 0 \, , \, a_0 \geq 1 \, ,  \end{array} \right.
\end{displaymath}
where $x=\frac{\xi_1}{b} + \frac{\xi_2}{b^2} + \dots$.
\end{theorem}

We now briefly recall the CBC algorithm from \cite{BD10}, but immediately present it for general weights. The generating polynomials of the polynomial lattice rule are chosen from the following set
\begin{displaymath}
R_{b,M}:= \left\{ q \in \mathbb{Z}_b[x] : \deg(q) < M \textrm{ and } q \neq 0 \right\} \, ,
\end{displaymath}
so $\vert R_{b,M} \vert = b^M-1$. The CBC algorithm constructs a $\bsq \in R^{s_1}_{b,M}$ so that $B(\bsq , [e], \tilde{\bsgamma}^{(e)})$, $e=1,\dots,s_1$, converges at a rate of $O(n^{-3 + \delta})$, for any $\delta>0$.

\begin{algorithm}[h!] \label{algcbc}
\caption{CBC 1 algorithm}
\begin{algorithmic}[1]\label{algcbc1}
\REQUIRE $b$ a prime,  $s_1,M \in \nn$, an irreducible polynomial $p \in \zz_b[x]$ with $\deg(p)=M$, and weights $\left( \gamma_u \right)_{u \subseteq [s_1]}$. \STATE Set $q_1=1$. \FOR{$e=2$ to
$s_1$} \STATE find $q_e \in R_{b,M}$ by minimizing $B((q_1,\dots,q_e),[e], \tilde{\bsgamma}^{(e)})$ as a function of $q_e$. \ENDFOR \RETURN
$\bsq=(q_1,\dots,q_{s_1})$.
\end{algorithmic}
\end{algorithm}
The next theorem is the analogue of Theorem 1 in \cite{BD10}, but immediately presented for general weights.

\begin{theorem} \label{thepcbcgenweights} Let $(\gamma_u)_{u \in \mathcal{U}}$ be general weights. Assume that the vector $\bsq = (q_1, \dots, q_{s_1})$ is obtained using the CBC 1 algorithm, see Algorithm \ref{algcbc1}. Then we have for all $1 \leq \tau < 3$, $u \subseteq [s_1]$,
\begin{displaymath}
B( \bsq , [z_u] , \tilde{\bsgamma}^{(z_u)} ) \leq (b^M-1)^{- \tau} \left( \sum_{ z_u \in w  \subseteq [z_u] } \gamma^{\frac{1}{\tau}}_w C^{\vert w \vert}_{b, \frac{1}{\tau} } \right)^{\tau} \, ,
\end{displaymath}
where $C_{b,\lambda}$ is given by \eqref{eqCconstCBC} and $z_u = \max u$.
\end{theorem}

We now discuss how to extend the vector $\bsq =(q_1, \dots, q_{s_1})$ from Theorem \ref{thepcbcgenweights} to a vector $(q_1, \dots, q_{s_2})$, where $s_2 > s_1$. Intuitively speaking, we employ the polynomials $q_1, \dots, q_{s_1}$ constructed via the CBC 1 algorithm, and simply continue the CBC search, now constructing polynomials $q_{s_1+1}, \dots, q_{s_2}$. This is formalized in Algorithm \ref{algcbc21}, the \emph{CBC 2} algorithm.

\begin{algorithm}[h!] \label{algcbc2}
\caption{CBC 2 algorithm}
\begin{algorithmic}[1]\label{algcbc21}
\REQUIRE $b$ a prime, $s_1,s_2, M \in \nn$,  where $s_1 < s_2$, an irreducible polynomial $p \in \zz_b[x]$ with $\deg(p)=M$,  weights $\left( \gamma_u \right)_{u \subseteq [s_2]}$ and polynomials $q_j$, $j=1, \dots, s_1$. \FOR{$e=s_1+1$ to
$s_2$} \STATE find $q_e \in R_{b,M}$ by minimizing $B((q_1,\dots,q_e), [e], \tilde{\bsgamma}^{(e)})$ as a function of $q_e$. \ENDFOR \RETURN
$\bsq=(q_1,\dots,q_{s_2})$.
\end{algorithmic}
\end{algorithm}
We get the following corollary to Theorem \ref{thepcbcgenweights}, which shows that Algorithm \ref{algcbc21} achieves the essentially optimal rate of convergence.
\begin{corollary} \label{theocbcgensets} Let $(\gamma_u)_{u \in \mathcal{U}}$ be general weights. Assume that the polynomials $q_1, \dots, q_{s_1}$ are given and that the polynomials $q_{s_1+1} , \dots, q_{s_2}$ are obtained via Algorithm \ref{algcbc21}. Then we have for all $1 \leq \tau < 3$ and $u \subseteq [s_2]$,
\begin{displaymath}
B (\bsq, [z_u], \tilde{\bsgamma}^{(z_u)} ) \leq \frac{1}{(b^M-1)^{\tau}} \left( \sum_{ z_u \in w \subseteq [z_u] } \gamma^{\frac{1}{\tau}}_w C^{\vert w \vert}_{b, \frac{1}{\tau}} \right)^{\tau} \, ,
\end{displaymath}
where $C_{b, \lambda}$ is defined as in \eqref{eqCconstCBC} and $z_u = \max u$.
\end{corollary}

Due to Theorem \ref{thepcbcgenweights}, Corollary \ref{theocbcgensets}, and Equation \eqref{eqVarboundscrpolylattrule2}, our approach provides us with a scrambled polynomial lattice rule $\left\{ \bsx_i \right\}^{b^M}_{i=1} \in [0,1)^{v_k}$, such that the algorithm
\begin{displaymath}
Q_{v_k,n_k}(f) = \frac{1}{n_k} \sum^{n_k}_{i=1} f(\bsx_i) \,
\end{displaymath}
satisfies the claims made in Theorem \ref{theoCBC}.


\end{document}